\documentclass[11pt,a4paper]{article}

\usepackage[a4paper, margin=2.5cm]{geometry}
\usepackage{graphicx}
\usepackage{multirow}
\usepackage{amsmath,amssymb,amsfonts}
\usepackage{amsthm}
\usepackage{mathrsfs}
\usepackage[title]{appendix}
\usepackage{xcolor}
\usepackage{textcomp}
\usepackage{booktabs}
\usepackage{tabularx}
\usepackage{subcaption}
\usepackage{hyperref}
\usepackage{url}
\usepackage{natbib}
\usepackage{rotating}
\usepackage{iftex}
\ifpdftex
  \usepackage{epstopdf}
\fi

\graphicspath{{images/}}

\theoremstyle{plain}
\newtheorem{theorem}{Theorem}
\newtheorem{proposition}[theorem]{Proposition}

\theoremstyle{definition}

\newtheorem{remark}{Remark}

\begin{document}

\title{Comparing Physics-Informed and Neural ODE Approaches for Modeling Nonlinear Biological Systems:
A Case Study Based on the Morris--Lecar Model}

\author{Nikolaos M.\ Matzakos\thanks{%
  Department of Education,
  School of Pedagogical and Technological Education (ASPETE),
  141\,21 Athens, Greece.
  E-mail: \texttt{nikmatz@aspete.gr}.
  ORCID: \texttt{0000-0001-8647-6082}.}
\and
Chrisovalantis Sfyrakis\thanks{%
  Department of Mechanical Engineering Educators,
  School of Pedagogical and Technological Education (ASPETE),
  151\,22 Athens, Greece.
  E-mail: \texttt{hammer@aspete.gr}.}}

\date{}

\maketitle

\footnotetext{This is a preprint version of a manuscript currently under review.}

\begin{abstract}
Physics-Informed Neural Networks (PINNs) and Neural Ordinary Differential Equations (NODEs) represent two distinct machine learning frameworks for modeling nonlinear neuronal dynamics. This study systematically evaluates their performance on the two-dimensional Morris--Lecar model across three canonical bifurcation regimes: Hopf, Saddle-Node on Limit Cycle, and homoclinic orbit. Synthetic time-series data are generated via numerical integration under controlled conditions, and training is performed using collocation points for PINNs and adaptive solvers for NODEs (Dormand--Prince method).

PINNs incorporate the governing differential equations into the loss function using automatic differentiation, which enforces physical consistency during training. In contrast, NODEs learn the system's vector field directly from data, without prior structural assumptions or inductive bias toward physical laws. Model performance is assessed using standard regression metrics, including Mean Squared Error (MSE), Mean Absolute Error (MAE), Mean Absolute Percentage Error (MAPE), and the coefficient of determination.

Results indicate that PINNs tend to achieve higher accuracy and robustness in scenarios involving stiffness or sensitive bifurcations, owing to their embedded physical structure. NODEs, while more expressive and flexible, operate as black-box approximators without structural constraints, which can lead to reduced interpretability and stability in these regimes. Although advanced variants of NODEs (e.g., ANODEs, latent NODEs) aim to mitigate such limitations, their performance under stiff dynamics remains an open question. These findings emphasize the trade-offs between physics-informed models, which embed structure and interpretability, and purely data-driven approaches, which prioritize flexibility at the cost of physical consistency.
\end{abstract}



\section{Introduction}\label{sec1}
Mathematical modeling of biological neurons is essential for understanding how electrical signals propagate in excitable cells. Among the foundational models in computational neuroscience \cite{gasparinatou2022}, the Morris–Lecar model \cite{morris1981} is widely recognized for its balance between biophysical realism and mathematical tractability. It describes membrane voltage dynamics through a system of nonlinear ordinary differential equations and is frequently used to test machine learning–based solvers for time-dependent biological systems.
Recent advances in deep learning have introduced data-driven frameworks for solving differential equations, with two notable approaches emerging: Neural Ordinary Differential Equations (NODEs) and Physics-Informed Neural Networks (PINNs). Both utilize neural networks, but with distinct philosophies. NODEs learn the underlying dynamics of a system by approximating the vector field directly from observed time-series data \cite{chen2019, dupont2019, massaroli2021}. In contrast, PINNs encode prior knowledge of the governing differential equations and any initial or boundary conditions into the loss function, enabling solution learning through automatic differentiation \cite{raissi2019, karniadakis2021, lu2021}.
Despite their widespread use, direct empirical comparisons between NODEs and PINNs on the same dynamical system remain limited. This study addresses that gap by systematically applying both methods to the Morris–Lecar model and evaluating their performance across five key dimensions: predictive accuracy, data efficiency, noise robustness, computational cost, and interpretability of the learned dynamics.
The Morris–Lecar system provides an ideal benchmark for such a comparison. Its low dimensionality, analytically defined equations, and rich dynamical behavior—including oscillations, excitability, and bifurcations—enable a rigorous evaluation of modeling capabilities. By testing NODEs and PINNs under identical experimental conditions, we explore how each framework handles the complexities of biologically inspired nonlinear dynamics.
Beyond this focused comparison, our findings contribute to the broader discussion around hybrid modeling approaches that combine data-driven learning with embedded physical constraints, such as latent PINNs or neural operators. As the field moves toward more generalizable and robust solvers for dynamical systems, systematic benchmarking on biologically grounded models remains a critical step.
\section{Background}
\label{sec2}

\subsection{The Morris--Lecar Model}
\label{sec2.1}

We consider the standard two-dimensional version of the Morris–Lecar~\cite{morris1981} model, which, due to its low dimensionality and rich nonlinear behavior, has recently been explored as a benchmark for learning dynamical systems from data, especially in computational neuroscience. The Morris--Lecar model describes membrane voltage oscillations in barnacle muscle fibers and has become a canonical example of a two-dimensional neuronal oscillator. It is governed by the nonlinear system:
\begin{equation}
\begin{cases}
C \dfrac{dV}{dt} = -g_{\text{Ca}} M_\infty(V)(V - V_{\text{Ca}}) - g_K N (V - V_K) - g_L (V - V_L) + I, \\[6pt]
\dfrac{dN}{dt} = \varphi \, \dfrac{N_\infty(V) - N}{\tau_N(V)}
\end{cases}
\label{eq1}
\end{equation}
where $V(t)$ is the membrane potential, $N(t)$ is the potassium gating variable, $C$ is the membrane capacitance, $g_{\text{Ca}}, g_K, g_L$ are maximal conductances, and $V_{\text{Ca}}, V_K, V_L$ are the reversal potentials for calcium, potassium, and leak currents. The external current $I$ acts as a bifurcation parameter.

The voltage-dependent functions are defined as:
\begin{align}
M_\infty(V) &= \frac{1}{2} \left[ 1 + \tanh\left( \frac{V - V_1}{V_2} \right) \right], \label{eq:Minf} \\
N_\infty(V) &= \frac{1}{2} \left[ 1 + \tanh\left( \frac{V - V_3}{V_4} \right) \right], \label{eq:Ninf} \\
\tau_N(V) &= \frac{1}{\cosh\left( \frac{V - V_3}{2 V_4} \right)}. \label{eq:tauN}
\end{align}
This model belongs to the class of fast--slow systems \cite{fenichel1979,ermentrout2010, li2011}, as the voltage $V$ evolves much faster than $N$ due to the small parameter $\varphi \ll 1$. This time-scale separation enables bifurcation and singular perturbation analysis. Let us denote the system compactly as
\begin{equation} 
\frac{d}{dt}
\begin{pmatrix}
V \\
N
\end{pmatrix}
= F(V,N),
\end{equation}
with $F: \mathbb{R}^2 \to \mathbb{R}^2$ a smooth vector field.

Steady states $(V^*, N^*)$ satisfy $F(V^*, N^*) = 0$, or equivalently:
\begin{equation*}
\left\{
\begin{aligned}
0 &= -g_{\text{Ca}} M_\infty(V^*)(V^* - V_{\text{Ca}}) - g_K N^* (V^* - V_K) - g_L (V^* - V_L) + I, \\
N^* &= N_\infty(V^*)
\end{aligned}
\right.
\end{equation*}
To justify the subsequent stability and learning analysis, we first confirm the existence of an equilibrium point for the system.
\begin{proposition}
\label{prop2.1}
Consider the Morris--Lecar model as given in equation~(1), with a fixed external input current $I \in \mathbb{R}$ and continuous functions $M_\infty(V)$, $N_\infty(V)$, and $\tau_N(V)$. Then, the system admits at least one equilibrium point $(V^*, N^*) \in \mathbb{R}^2$.
\end{proposition}

\begin{proof}[Sketch of proof]
At equilibrium, the time derivatives in equation~(1) vanish, yielding the algebraic conditions
\[
\begin{cases}
0 = -g_{Ca} M_\infty(V)(V - V_{Ca}) - g_K N(V - V_K) - g_L(V - V_L) + I, \\
N = N_\infty(V).
\end{cases}
\]
Substituting the second equation into the first gives a scalar nonlinear equation for $V$:
\[
f(V) := g_{Ca} M_\infty(V)(V - V_{Ca}) + g_K N_\infty(V)(V - V_K) + g_L(V - V_L) = I.
\]
The functions $M_\infty(V)$ and $N_\infty(V)$ are smooth sigmoidal maps, hence continuous and bounded on $\mathbb{R}$, so $f(V)$ is continuous. Moreover, due to the linear leak term, we have $\lim_{V \to -\infty} f(V) = -\infty$ and $\lim_{V \to +\infty} f(V) = +\infty$. By the Intermediate Value Theorem, there exists $V^* \in \mathbb{R}$ such that $f(V^*) = I$. Setting $N^* := N_\infty(V^*)$, we obtain an equilibrium point $(V^*, N^*)$.
\end{proof}
See~\cite{izhikevich2006, ermentrout2010, azizi2020, li2011} for related discussions on equilibrium points and bifurcations in planar conductance-based neuron models. The behavior near equilibrium points is crucial for understanding the phase space geometry, which directly impacts the performance and learnability of PINNs and NODEs trained on such systems.

To analyze the local stability of the Morris–Lecar system around a fixed point $(V^*, N^*)$, we compute the Jacobian matrix of the vector field. Let $f_1(V, N)$ and $f_2(V, N)$ denote the right-hand sides of the system equations. The Jacobian matrix is defined as:
\begin{equation}
J = 
\begin{bmatrix}
\frac{\partial f_1}{\partial V} & \frac{\partial f_1}{\partial N} \\
\frac{\partial f_2}{\partial V} & \frac{\partial f_2}{\partial N}
\end{bmatrix}.
\label{eq:jacobian_def}
\end{equation}
The dynamical equations are given by:
\[
\begin{aligned}
f_1(V, N) &= \frac{1}{C} \left( -g_{\text{Ca}} M_\infty(V)(V - V_{\text{Ca}}) - g_K N(V - V_K) - g_L(V - V_L) + I \right), \\
f_2(V, N) &= \varphi \, \frac{N_\infty(V) - N}{\tau_N(V)}.
\end{aligned}
\]
We now compute the partial derivatives and the Jacobian becomes:
\begin{equation}
J =
\begin{bmatrix}
\displaystyle \frac{1}{C} \left( -g_{\text{Ca}} M'_\infty(V)(V - V_{\text{Ca}}) - g_{\text{Ca}} M_\infty(V) - g_K N - g_L \right) & \displaystyle \frac{-g_K(V - V_K)}{C} \\[8pt]
\displaystyle \frac{\varphi}{\tau_N(V)} N'_\infty(V) & \displaystyle -\frac{\varphi}{\tau_N(V)}
\end{bmatrix}.
\label{eq:jacobian_final}
\end{equation}

Since the system is planar, the nature of the fixed point is determined by the eigenvalues of $J$, which in turn depend on its trace and determinant. The fixed point is locally asymptotically stable if:
\begin{equation}
\operatorname{Tr}(J) < 0 \quad \text{and} \quad \det(J) > 0.
\end{equation}
\begin{remark}[Remark on fast–slow structure.]
\label{rem1}
Although we do not explicitly employ geometric singular perturbation theory, it is worth noting that the Morris--Lecar model possesses an intrinsic fast–slow structure due to the small parameter $\varphi$ in the recovery variable equation. This timescale separation gives rise to geometric features such as slow manifolds and stiff transients, particularly near bifurcation boundaries. These properties, as studied in classical dynamical systems literature~\cite{fenichel1979,azizi2020,paraskevov2021}, help explain qualitative behaviors such as long quiescent phases or sharp transitions, and partially account for the increased difficulty of learning dynamics in certain regimes using data-driven methods such as PINNs and NODEs.
\end{remark}
\subsection{Physics-Informed Neural Networks}
\label{sec2.2}
A class of deep learning models integrates prior knowledge of physical systems directly into the neural network training process by embedding governing differential equations into the loss function—these models are known as PINNs. Introduced by \cite{raissi2019}, PINNs leverage the structure of the underlying physical laws—typically ordinary or partial differential equations—as soft constraints by incorporating them into the optimization objective. This strategy transforms the learning task from a purely data-fitting problem into a constrained optimization grounded in the governing physics.
Formally, given a system of differential equations with a known structure but potentially unknown solution, a neural network is trained not only to fit observed data but also to minimize the residuals of the governing equations at selected points in the domain. The resulting composite loss function typically takes the form:
\begin{equation}
	L\left(\theta\right)=\left|\left|u_\theta\left(t_i\right)-u\left(t_i\right)\right|\right|^2+\lambda\left|\left|\mathcal{N}\left[u_\theta\left(t\right)\right]\right|\right|^2     
\end{equation}
\begin{center}
\mbox{[data loss]+ [ physics loss]}
\end{center}
where $\mathcal{N}\left[\cdot\right]$ denotes the differential operator representing the governing equations and $\lambda$ is a scalar weighting hyperparameter balancing data fidelity against physical compliance.
 A key innovation of this framework lies in the evaluation of the physics loss via automatic differentiation \cite{raissi2019}, which enables the computation of derivatives of the neural network outputs with respect to their inputs. This mechanism allows PINNs to calculate ODE or PDE residuals even when no explicit data for derivatives is available, offering the ability to interpolate and extrapolate in data-sparse regimes by enforcing physical structure.
Beyond their use in forward problems (e.g., predicting system trajectories), PINNs have been applied to inverse problems, such as estimating unknown parameters embedded in the governing equations. For instance, \cite{tartakovsky2020} demonstrated how physical parameters in subsurface flow could be treated as trainable variables within the network, facilitating simultaneous identification of hidden dynamics and parameters.
Another crucial design element in PINNs is the selection of collocation points—discrete locations in the domain where the residuals of the governing equations are enforced \cite{raissi2019}, \cite{lu2021}. These can coincide with observed data points or be sampled independently, using strategies such as uniform sampling or Latin Hypercube Sampling. Poorly chosen collocation points can lead to slow convergence or inaccurate approximations, highlighting the importance of spatial and temporal coverage.
To address challenges in training stability and convergence, recent work has proposed several improvements. For example, \cite{wang2022} investigated gradient flow pathologies in PINNs and introduced adaptive weighting strategies and curriculum learning methods to better balance data and physics terms throughout training. These approaches help mitigate issues such as stiffness and imbalanced losses that are common in multi-component loss functions.
PINNs have also demonstrated success in biomedical applications. A notable example is the work of \cite{kissas2020}, where PINNs were used to reconstruct arterial blood pressure fields from 4D flow MRI velocity data in cardiovascular modeling. This application highlights the ability of PINNs to integrate sparse, noisy, or incomplete measurements with physically informed constraints to enable robust simulation of physiological systems.
While PINNs offer promising accuracy, generalization, and interpretability, recent comparisons with emerging operator learning methods have illuminated important trade-offs. Methods such as the Fourier Neural Operator \cite{li2021} and DeepONet \cite{lu2021} have shown superior generalization performance across diverse families of parametric PDEs, particularly in high-dimensional settings. However, these methods tend to forgo explicit physical structure in favor of learning mappings between function spaces, often sacrificing interpretability and physical consistency. This contrast emphasizes the ongoing tension between generalization capability and mechanistic fidelity in scientific machine learning.

\subsection{Neural ODEs}\label{sec2.3}
Introduced by Chen in \cite{chen2019}, this framework constitutes a continuous-depth extension of residual neural networks (ResNets), where the transformation of hidden states is governed by a parameterized ordinary differential equation. Instead of stacking discrete layers, NODEs model the evolution of a hidden state as the solution to an ordinary differential equation parameterized by a neural network. This formulation allows NODEs to represent continuous-time dynamics directly and adaptively.
In mathematical terms, the hidden state $h\left(t\right)$ evolves according to:
\begin{equation} 
	\frac{dy}{dt}=f_\theta\left(y\left(t\right),t\right)	
\end{equation}
where ${\frac{dy}{dt}=f}_\theta$ is a neural network with parameters $\theta$, and the output $h\left(t\right)$ obtained via numerical integration using methods such as Runge–Kutta or adaptive solvers. Training is performed by minimizing the mismatch between predicted trajectories and observed data, typically through the adjoint sensitivity method which allows memory-efficient backpropagation through the ODE solver.
NODEs are particularly well-suited for modeling systems where the underlying vector field is unknown or analytically intractable. Their continuous formulation allows smooth interpolation between observations and naturally accommodates irregularly sampled or non-uniform time series.
One of the core strengths of NODEs lies in their flexibility: since the model makes no assumptions about the governing equations of the system, it can be trained entirely from data. This makes NODEs especially attractive for black-box modeling of dynamical systems. However, the lack of embedded physical structure can also be a limitation. NODEs may struggle with extrapolation beyond the training domain and often require dense, low-noise data to achieve accurate results. 
Several extensions have been proposed to improve the expressiveness and robustness of NODEs. Augmented NODEs \cite{dupont2019} expand the dimensionality of the hidden state to capture more complex dynamics, while Continuous Normalizing Flows (CNFs) adapt the NODE framework for density estimation and generative modeling. Another notable variant, ANODEV2 \cite{gholami2021}, introduces architectural refinements and coupling mechanisms to address stiffness and gradient instability during training.
Despite these improvements, training NODEs remains computationally expensive due to the repeated need to solve differential equations during optimization. The adjoint method, although efficient in memory usage, can introduce numerical challenges, particularly in stiff regimes.
Overall, NODEs offer a powerful and flexible framework for learning continuous-time dynamics directly from data, especially in domains where physical models are unknown or difficult to specify. Their modular structure and compatibility with modern deep learning tools have made them popular across a range of applications, including neuroscience, finance, and climate modeling.
\section{Methodology}
\label{sec3}
This section outlines the experimental design used to compare PINNs and NODEs on the Morris–Lecar neuronal model. We begin by motivating the selection of the Morris–Lecar system as a benchmark, followed by a description of the dynamical regimes explored in this study and the procedure for generating synthetic training data. We evaluate the predictive performance of the models using standard regression metrics, including the Mean Squared Error (MSE), Mean Absolute Error (MAE), Root Mean Squared Error (RMSE), the coefficient of determination ($R^2$), and others. A detailed description and the mathematical definitions of these metrics are provided in Section~\ref{sec:metrics}.

\subsection{Rationale for Model Selection}
\label{sec3.1}
The Morris–Lecar model provides an ideal testbed for comparing data-driven and physics-informed neural modeling approaches due to its biophysical interpretability, well-established mathematical formulation, and capacity to exhibit a rich set of nonlinear behaviors. Table~\ref{tab:just} summarizes the main characteristics of the model and their relevance to the PINN vs. NODE comparison.
\begin{table}[!htb]
\caption{Justification for using the Morris--Lecar model as a benchmark in the comparison of PINNs and NODEs.}
\label{tab:just}
\begin{tabularx}{\textwidth}{@{}lX X@{}}
\toprule
\textbf{Criterion} & \textbf{Morris--Lecar Model Characteristics} & \textbf{Relevance to PINNs \& NODEs} \\
\midrule
Known physical equations        & Analytically defined ODE system      & Enables direct embedding of physics in PINNs; NODEs learn dynamics from data \\
Nonlinear dynamics              & Exhibits excitability and oscillations     & Tests the models' ability to capture bifurcation behavior \\
Ground truth availability       & Numerical integration feasible       & Allows quantitative performance evaluation (Subsection~\ref{sec:metrics}) \\
Low-dimensional yet rich        & Two-variable system with complex transitions & Suitable for compact yet expressive architectures \\
Biophysical interpretability    & Neurophysiologically meaningful parameters & PINNs offer insight; NODEs assess black-box adaptability \\
\bottomrule
\end{tabularx}
\end{table}
\subsection{Dynamical Regimes of the Morris–Lecar Model}
\label{sec3.2}
In this study, we examine three distinct parameter configurations of the Morris–Lecar model, each representing a different type of neuronal excitability: Hopf Bifurcation, Snlc and Homoclinic Orbit. These configurations are drawn from the canonical formulations described in “Mathematical Foundations of Neuroscience” by Ermentrout and Terman in \cite{ermentrout2010} and Azzizi et. al \cite{azizi2020}, where the Morris–Lecar model is used to explore bifurcation structures that give rise to different classes of neural responses.
By using these three well-established regimes, we aim to assess how well PINNs and NODEs can learn or reconstruct the underlying system dynamics under qualitatively different behaviors—ranging from sharp all-or-none firing to subthreshold oscillations
\subsection{Data Generation and Regime Characterization}
\label{sec3.3}
To evaluate the modeling capabilities of PINNs and NODEs, we generated synthetic datasets based on numerical simulations of the Morris–Lecar system. The simulations were performed using a custom Python script (\texttt{synthetic\_M\_L\_data.py}) that numerically integrates the system's equations using the \texttt{solve\_ivp} function from \texttt{scipy.integrate}. For each simulation, we collected trajectories of the membrane potential \( V(t) \) and the recovery variable \( N(t) \) over the time interval \( t \in [0,\ 300] \,\text{ms} \), using 3000 uniformly spaced time points. These trajectories form the ground truth used for training and evaluating the models.

Simulations were conducted for the three bifurcation regimes, using parameter values listed in Table~\ref{tab:parameter}, adapted from \cite{azizi2020}. For each regime, we performed a current sweep over the interval \( I_{\text{ext}} \in [0, 120]\,\mu\text{A}/\text{cm}^2 \). The resulting peak-to-peak voltage amplitudes were used to construct bifurcation diagrams, shown in Figure~\ref{fig1}, which guided the selection of representative current values for training and testing.

\begin{table}[!htb]
\caption{Parameter values for the Morris–Lecar model under three canonical bifurcation regimes (\cite{azizi2020}).}
\label{tab:parameter}
\centering
\begin{tabular}{@{}llll@{}}
\toprule
\textbf{Parameter} & \textbf{Hopf} & \textbf{SNLC} & \textbf{Homoclinic} \\
\midrule
$\varphi$    & 0.04  & 0.067 & 0.23 \\
$g_{Ca}$  & 4.4   & 4     & 4 \\
$V_3$     & 2     & 12    & 12 \\
$V_4$     & 30    & 17.4  & 17.4 \\
$E_{Ca}$  & 120   & 120   & 120 \\
$E_K$     & -84   & -84   & -84 \\
$E_L$     & -60   & -60   & -60 \\
$g_K$     & 8     & 8     & 8 \\
$g_L$     & 2     & 2     & 2 \\
$V_1$     & -1.2  & -1.2  & -1.2 \\
$V_2$     & 18    & 18    & 18 \\
$C_M$     & 20    & 20    & 20 \\
\bottomrule
\end{tabular}
\end{table}

\begin{figure}[!htbp]
    \centering
    \begin{subfigure}[t]{0.48\textwidth}
        \centering
        \includegraphics[width=\textwidth]{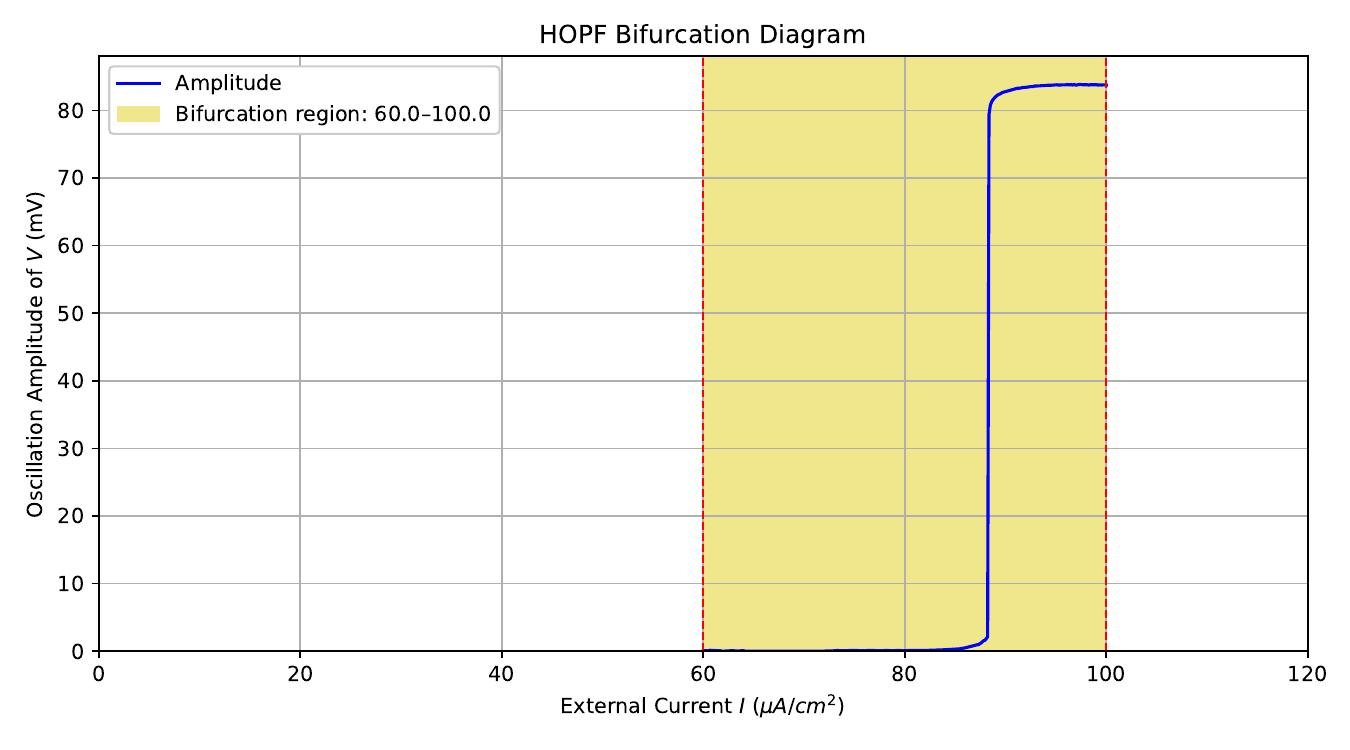}
        \caption{Hopf regime.}
    \end{subfigure}
    \hfill
    \begin{subfigure}[t]{0.48\textwidth}
        \centering
        \includegraphics[width=\textwidth]{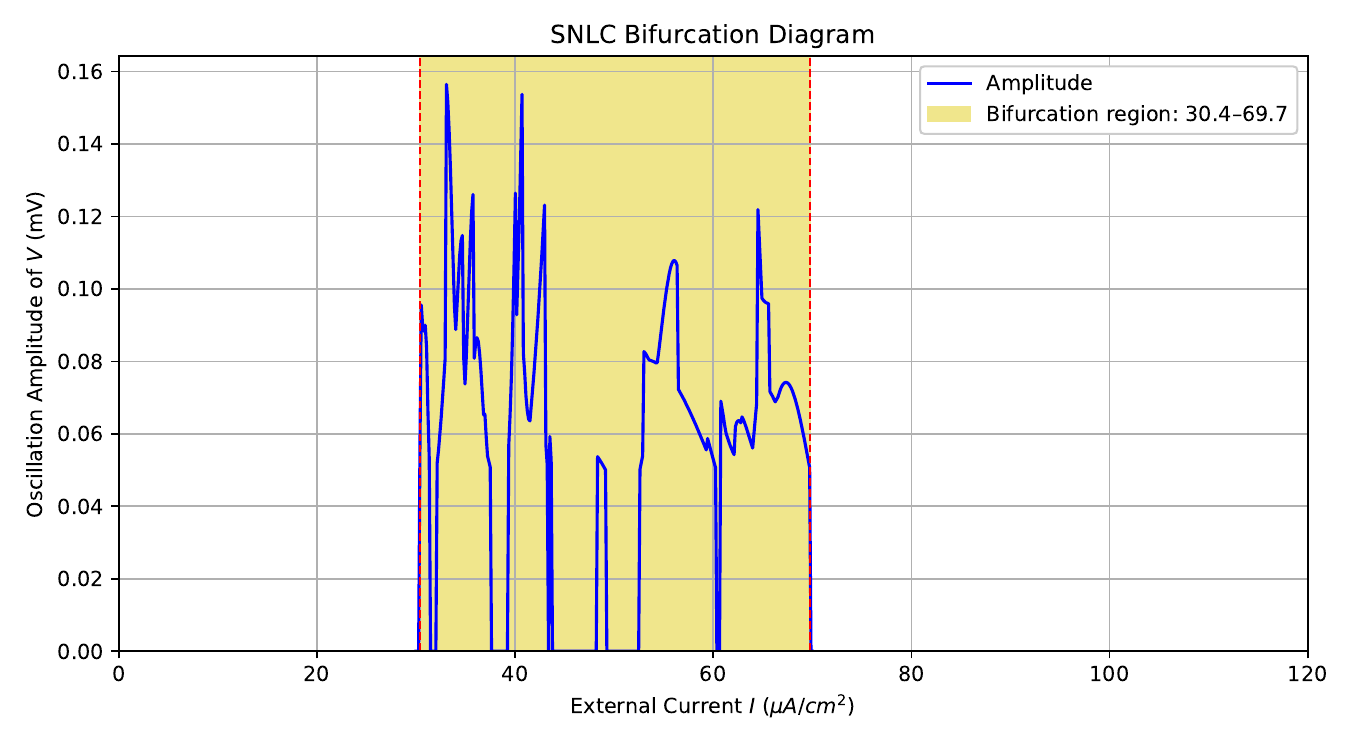}
        \caption{SNLC regime.}
    \end{subfigure}
    \begin{subfigure}[t]{0.48\textwidth}
        \centering
        \includegraphics[width=\textwidth]{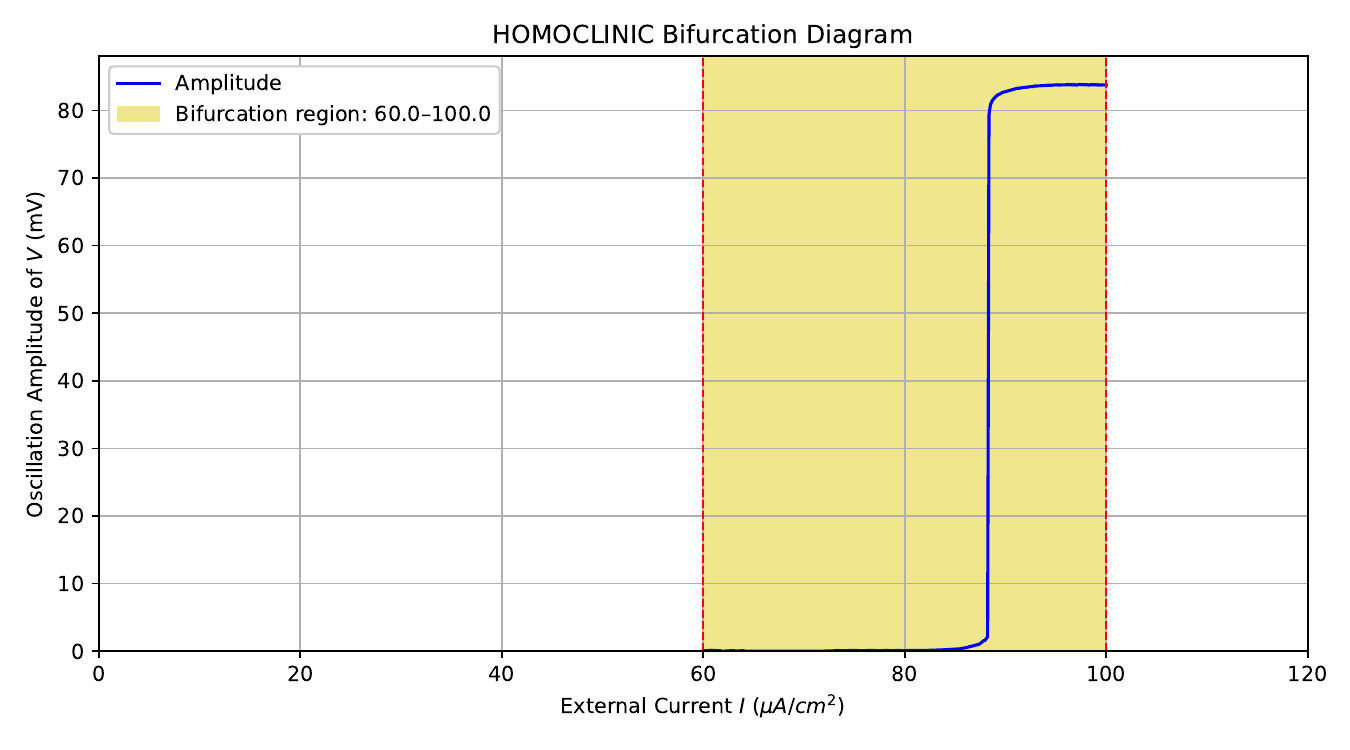}
        \caption{Homoclinic regime.}
    \end{subfigure}
		\caption{Bifurcation diagrams of the Morris–Lecar system under the three regimes defined in Table~\ref{tab:parameter}, obtained by sweeping $I_{\text{ext}}$ over the range $[0,\ 120]\,\mu\text{A}/\text{cm}^2$.}
    \label{fig1}
\end{figure}

Based on these diagrams, we selected representative current values within the bifurcation regions to highlight the qualitative transitions in the system’s dynamics. Specifically, we used \( I_{\text{ext}} = 90\,\mu\text{A}/\text{cm}^2 \) for the Hopf regime, \( 42\,\mu\text{A}/\text{cm}^2 \) for the SNLC regime, and \( 50\,\mu\text{A}/\text{cm}^2 \) for the Homoclinic regime. These trajectories serve as the input data for downstream training.

Figure~\ref{fig2} presents the corresponding phase portraits and equilibrium structures, while Table~\ref{tab:eq_analysis} summarizes the fixed points, Jacobians, and eigenvalues, offering additional insight into local stability.

\begin{figure}[!htbp]
    \centering
    \begin{subfigure}[t]{0.48\textwidth}
        \centering
        \includegraphics[width=\textwidth]{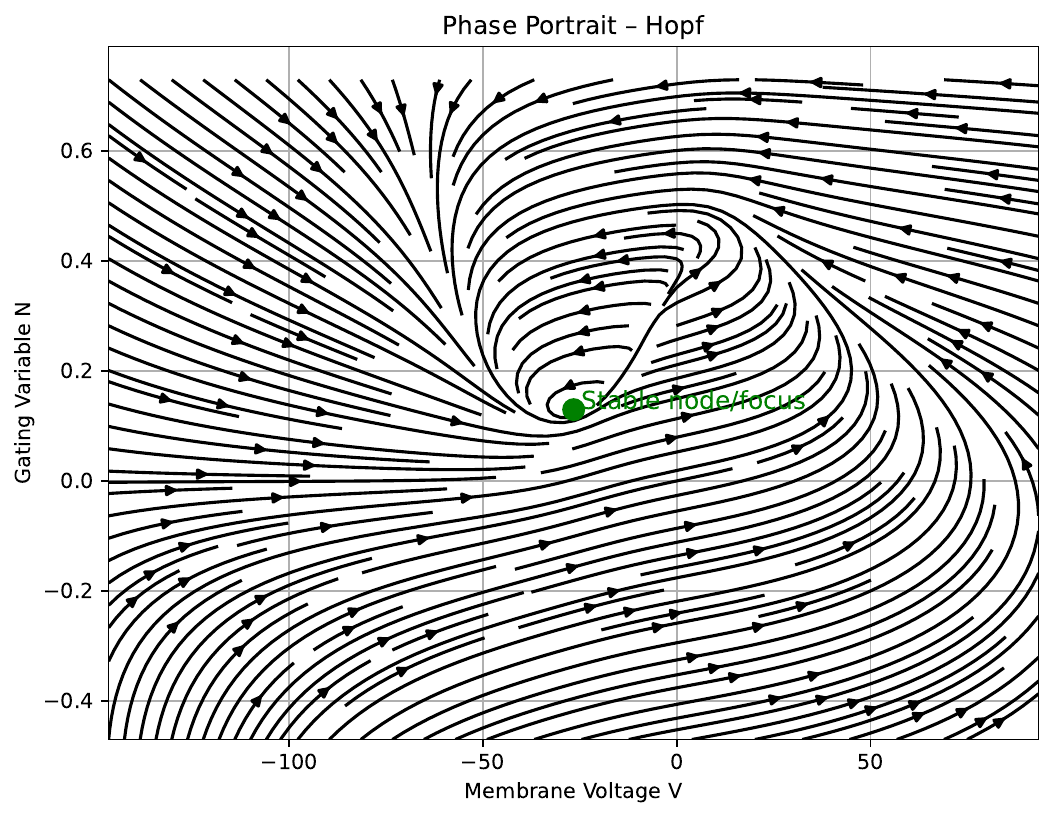}
        \caption{Hopf regime.}
    \end{subfigure}
    \hfill
    \begin{subfigure}[t]{0.48\textwidth}
        \centering
        \includegraphics[width=\textwidth]{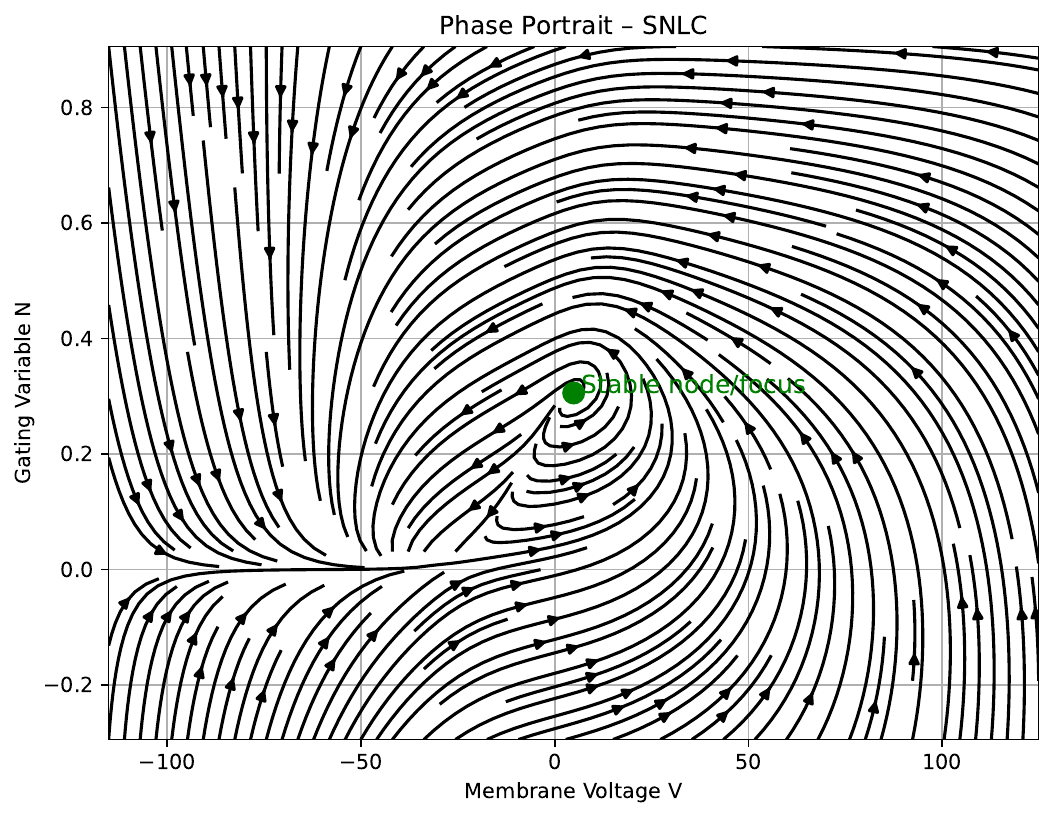}
        \caption{SNLC regime.}
    \end{subfigure}
    \begin{subfigure}[t]{0.48\textwidth}
        \centering
        \includegraphics[width=\textwidth]{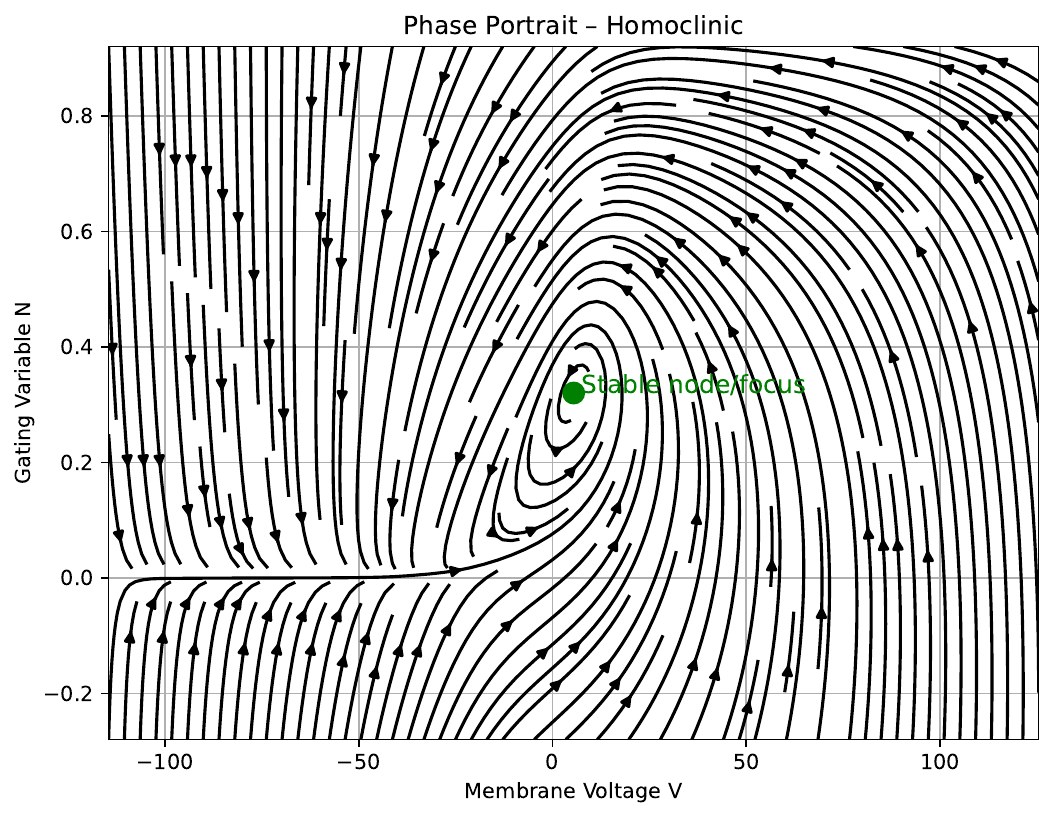}
        \caption{Homoclinic regime.}
    \end{subfigure}
    \caption{Equilibrium analysis and phase portraits for each bifurcation regime under selected representative currents.}
    \label{fig2}
\end{figure}

\begin{table}[!htbp]
\centering
\caption{Equilibrium properties of the Morris–Lecar system under representative currents.}
\label{tab:eq_analysis}
\renewcommand{\arraystretch}{1.3}
\begin{tabular}{lccc}
\toprule
\textbf{Scenario} & \textbf{Homoclinic} & \textbf{SNLC} & \textbf{Hopf} \\
\midrule
Equilibrium Point $(V^*, N^*)$ &
(5.4540, 0.3203) &
(4.8622, 0.3057) &
(-26.5969, 0.1294) \\
Jacobian &
$\begin{bmatrix}
-35.588 & -35.782 \\
0.005655 & -0.2341
\end{bmatrix}$ &
$\begin{bmatrix}
-35.327 & -35.545 \\
0.001601 & -0.0684
\end{bmatrix}$ &
$\begin{bmatrix}
-22.935 & -22.961 \\
0.000269 & -0.0446
\end{bmatrix}$ \\
Eigenvalues &
$-35.583,\ -0.240$ &
$-35.326,\ -0.070$ &
$-22.935,\ -0.045$ \\
Type &
Stable node/focus &
Stable node/focus &
Stable node/focus \\
\bottomrule
\end{tabular}
\end{table}
\clearpage
\subsection{Implementation of PINN}\label{sec3.4}

Building on the theoretical overview of PINNs presented in Section~\ref{sec2.2}, we implemented the Morris–Lecar model using the framework proposed by Raissi et al.~\cite{raissi2019}. The model approximates the system’s dynamics by minimizing a composite loss that balances data fidelity and compliance with the governing differential equations.

The approximation is carried out using a fully connected neural network in \texttt{PyTorch}~\cite{paszke2019}. The input is a scalar time variable $t$, and the outputs are predictions for the system's state variables: $\hat{\mathbf{y}}_\theta(t) = \left[ \hat{V}_\theta(t), \hat{N}_\theta(t) \right].$
The network architecture consists of three hidden layers with 128 neurons and \texttt{Tanh} activations, followed by a linear output layer. This configuration provides sufficient capacity for learning smooth time-dependent functions~\cite{hornik1989}. The total number of trainable parameters is $33,538$, and a single forward pass requires approximately $33,152$ floating points operations (FLOPs), estimated using the \texttt{THOP} library.

The key feature of the PINN approach is the incorporation of the system's physics into the training process. The derivatives $\frac{d\hat{V}}{dt}$ and $\frac{d\hat{N}}{dt}$ are computed using automatic differentiation (\texttt{torch.autograd.grad}) and are substituted into the differential equations to form the residual terms:
\begin{equation}
\begin{aligned}
f_1(t) &= C \frac{d\hat{V}}{dt} + g_{\text{Ca}} M_\infty(\hat{V})(\hat{V} - V_{\text{Ca}}) + g_K \hat{N} (\hat{V} - V_K) + g_L (\hat{V} - V_L) - I, \\
f_2(t) &= \frac{d\hat{N}}{dt} - \varphi \cdot \frac{N_\infty(\hat{V}) - \hat{N}}{\tau_N(\hat{V})},
\end{aligned}
\end{equation}
where $M_\infty, N_\infty$, and $\tau_N$ are voltage-dependent nonlinear functions defined via hyperbolic expressions.

The loss function is defined as the unweighted sum of a data loss (mean squared error between predicted and reference trajectories) and a residual loss derived from the above equations. Gradients are computed with respect to uniformly sampled time values $t$ with \texttt{requires\_grad=True}, enabling full backpropagation through the computational graph. The full implementation was developed in the pinn.py script, which contains the training, evaluation, and visualization routines for the PINN model.
\subsection{Implementation of NODE}\label{sec3.5}

To simulate the dynamics of the Morris--Lecar system using NODEs (\ref{sec2.2}), we implemented a scenario-independent framework in \texttt{PyTorch}~\cite{paszke2019}, utilizing the \texttt{torchdiffeq} library~\cite{chen2018b}. This framework learns the vector field of the underlying dynamical system directly from time series data, without requiring any explicit knowledge of the governing equations. The same implementation is applied across all bifurcation regimes (Hopf, SNLC, Homoclinic), enabling consistent evaluation of the method's ability to capture nonlinear behavior.

At the core of the NODE model lies a parameterized function $f_\theta(y)$, where $y = [V, N]$ denotes the system's state vector. The function $f_\theta$ is approximated by a feedforward neural network comprising three hidden layers of 128 neurons each with \texttt{Tanh} activations and a final output layer of two neurons representing the learned derivatives $\frac{dV}{dt}$ and $\frac{dN}{dt}$. This architecture was selected to balance expressiveness with computational efficiency~\cite{rackauckas2021}.

Model integration is performed using the Dormand--Prince 5\textsuperscript{th}-order adaptive solver (\texttt{method='dopri5'}) as provided by the \texttt{odeint} function in the \texttt{torchdiffeq} library. During training, the neural network is queried at the initial condition and integrated over the normalized time domain. The resulting trajectory is compared with ground truth data using the MSE loss function.

In each experiment, inverse normalization is applied to restore predictions to their original scale, and post-training analysis includes the generation of $V(t)$ and $N(t)$ time series plots, phase portraits, and loss curves. The number of trainable parameters and the number of floating-point operations (FLOPs) are also computed using the \texttt{THOP} library. All outputs, including trained model weights and performance metrics, are systematically saved in dedicated folders to ensure reproducibility. The full implementation was developed in the node.py script, which contains the training, evaluation, and visualization routines for the NODE model.

\subsection{Model Architecture and Experimental Setup }\label{sec3.6}	
To ensure a fair and systematic comparison, both the NODE and PINN models were trained on the same synthetic dataset generated from the 2D Morris--Lecar system. The dataset consists of $3000$ time points over the interval $t \in [0, 300]$ ms, with a step size of approximately $1$ ms. Each data point contains the membrane potential $V(t)$ and the recovery variable $N(t)$.

Both models were implemented in \texttt{PyTorch}~\cite{paszke2019}. The NODE implementation uses the \texttt{torchdiffeq} library~\cite{chen2018b} to solve ODEs via numerical integration, while the PINN approach leverages automatic differentiation (\texttt{torch.autograd.grad}) to enforce physics-based constraints through residual terms derived from the governing equations.

The models share a common architecture: fully connected feedforward networks with three hidden layers of $128$ neurons and \texttt{Tanh} activation functions. The NODE model receives the initial condition $[V_0, N_0]$ and integrates over time, while the PINN model uses scalar time input $t$ to predict $[\hat{V}(t), \hat{N}(t)]$ directly.

Normalization strategies differ. The NODE pipeline applies Min--Max normalization to $V$, $N$, and $t$, with inverse transformation after prediction. The PINN framework omits normalization to preserve the physical consistency of the residual terms.

Training is performed with the Adam optimizer and learning rate $10^{-3}$, across five epoch settings: $1000$, $2000$, $5000$, $10000$, and $20000$. The optimal model for each scenario is selected based on the minimal value of $\text{MSE}_V + \text{MSE}_N$. Computations were executed on an Apple M4 CPU with 32-bit floating-point precision. To ensure reproducibility, both pipelines enforce fixed seeds and deterministic algorithms using \texttt{torch.use\_deterministic\_algorithms}.

Evaluation metrics include MSE, RMSE, MAE, $R^2$, RMSPE, MAPE, and Max Error, computed separately for $V(t)$ and $N(t)$. Results include time-series plots, phase portraits, training loss curves, parameter counts, and FLOPs (via \texttt{THOP}). Outputs are organized into structured directories and stored in \texttt{.pth}, \texttt{.csv}, and \texttt{} formats.
In Table~\ref{tab:differences}, we present a summary of the methodological differences between PINNs and NODEs as implemented in this study.

\begin{table}[!htb]
\caption{Summary of Methodological Differences between PINNs and NODEs}
\label{tab:differences}
\begin{tabularx}{\textwidth}{@{}lX X X@{}}
\toprule
\textbf{Aspect} & \textbf{PINN} & \textbf{Neural ODE} \\
\midrule
Input–Output Structure & $t \rightarrow [V, N]$ & $[V_0, N_0] \rightarrow [V(t), N(t)]$ via ODE integration \\
Architecture           & FCNN (3 hidden layers × 128, Tanh) & FCNN (3 hidden layers × 128, Tanh or SiLU (see Remark~\ref{rem:SiLU})) \\
Physics Integration    & Explicit residuals via autograd     & None (data-driven only) \\
Normalization          & No normalization                   & Min–Max normalization with inverse scaling (see Remark~\ref{rem:minmax}) \\
Optimizer              & Adam (lr = $10^{-3}$)               & Adam (lr = $10^{-3}$) \\
Numerical Precision    & \texttt{float32}                    & \texttt{float32} \\
ODE Solver             & Not applicable                      & \texttt{odeint} (dopri5 integrator) \\
Parameter Count        & 33,538                         & 33,666 \\
FLOPs per Inference    & $\sim$33,152                       & $\sim$33,280 \\
Epochs Trained         & 1000 to 20000                       & 1000 to 20000 \\
Determinism            & Enabled & Enabled  \\
Evaluation Metrics     & MSE, RMSE, MAE, $R^2$, MAPE, RMSPE, Max Error & Same \\
\bottomrule
\end{tabularx}
\end{table}

\begin{remark}
\label{rem:minmax}
In the NODE implementation, we apply Min–Max normalization to $V$, $N$, and $t$ before training, mapping all values to $[0,1]$. This improves numerical stability and gradient propagation. After training, inverse normalization restores predictions to their physical scale. In contrast, PINN models operate on raw (unscaled) data to preserve the structure of the residuals derived from the governing differential equations. Normalizing $V$ would distort terms such as $M_\infty(V)$ and $\tau_N(V)$, compromising the model's physical interpretability.
\end{remark}

\begin{remark}
\label{rem:SiLU}
Following established practices in the literature, we use the hyperbolic tangent (Tanh) as the activation function in both PINN and Neural ODE architectures. Tanh has been widely adopted in PINNs due to its smoothness and ability to approximate continuous solutions effectively~\cite{jagtap2020}. Similarly, the original Neural ODE framework~\cite{chen2019} employed Tanh as its default choice. However, in the context of Neural ODEs, we also investigate the Sigmoid Linear Unit (SiLU or Swish), a smooth and non-monotonic activation that has demonstrated superior performance in deep learning models~\cite{ramachandran2017}. 
\end{remark}
To ensure a consistent and interpretable comparison between the two modeling approaches, we adopt a standard suite of regression-based metrics. These metrics, defined in the following subsection, allow us to evaluate the accuracy and generalization ability of each model across different dynamical regimes.

\subsection{Evaluation Metrics} \label{sec:metrics}

To systematically quantify model performance, we adopt standard regression-based metrics widely used in the evaluation of neural network surrogates for dynamical systems. These metrics offer complementary views on model accuracy, scale sensitivity, and generalization ability. In the context of biophysical systems like the Morris--Lecar model, it is essential to distinguish between absolute and relative errors, as well as to assess both average and worst-case deviations.

All metrics are computed between the predicted values $\hat{y}_i$ and the ground-truth targets $y_i$. For clarity, we report each metric separately for the membrane voltage $V(t)$ and the gating variable $N(t)$, using subscripts: e.g., $\mathrm{MSE}_V$, $\mathrm{MAPE}_N$, etc.

\begin{itemize}
    \item Mean Squared Error (MSE):
    \[
    \mathrm{MSE} = \frac{1}{n} \sum_{i=1}^{n} (y_i - \hat{y}_i)^2
    \]
    Captures the average squared deviation; sensitive to large errors.

    \item  Root Mean Squared Error (RMSE):
    \[
    \mathrm{RMSE} = \sqrt{\mathrm{MSE}}
    \]
    Expresses the error in the same units as the signal, facilitating physical interpretation.

    \item Mean Absolute Percentage Error (MAPE):
    \[
    \mathrm{MAPE} = \frac{100\%}{n} \sum_{i=1}^{n} \left| \frac{y_i - \hat{y}_i}{y_i} \right|
    \]
    A relative metric that expresses average error as a percentage. It may become unstable when $y_i \approx 0$.

    \item Coefficient of Determination ($R^2$):
    \[
    R^2 = 1 - \frac{\sum_{i=1}^{n}(y_i - \hat{y}_i)^2}{\sum_{i=1}^{n}(y_i - \bar{y})^2}
    \]
    Indicates the proportion of variance in the reference signal explained by the model. Values close to 1 denote high fidelity.

    \item Maximum Absolute Error (MaxErr):
    \[
    \mathrm{MaxErr} = \max_{i} |y_i - \hat{y}_i|
    \]
    Measures the worst-case absolute deviation across the predicted trajectory.

    \item Root Mean Squared Percentage Error (RMSPE):
    \[
    \mathrm{RMSPE} = \sqrt{ \frac{1}{n} \sum_{i=1}^{n} \left( \frac{y_i - \hat{y}_i}{y_i} \right)^2 }
    \]
    A scale-invariant metric that penalizes relative deviations; more robust than MAPE in certain regimes.
\end{itemize}

\section{Results and Comparison}\label{sec4}
In this section, we present the comparative performance of the PINNs and NODEs architectures in capturing the dynamics of the Morris–Lecar model under varying bifurcation scenarios. The evaluation focuses on both qualitative reproduction of voltage traces and quantitative error metrics.
\subsection{Qualitative Analysis of Model Predictions}\label{sec4.1}
To assess the ability of PINNs and NODEs to capture the dynamics of the Morris–Lecar model across different regimes, we evaluated the performance of both models on synthetic time series generated for representative values of external current $I$. 
In accordance with the existence of equilibria established in Proposition~\ref{prop2.1}, the Morris--Lecar system admits at least one fixed point for any biologically relevant set of parameters and for all values of the external current $I$. This theoretical guarantee provides a foundation for evaluating how well the PINN and NODE models capture the system’s behavior across different dynamic regimes. In particular, the ability of each model to converge to a stable resting state in the subcritical regime (e.g., $I = 50~\mu\text{A}/\text{cm}^2$ for the Hopf scenario), or to reproduce limit cycle oscillations beyond the bifurcation threshold (e.g., $I = 90~\mu\text{A}/\text{cm}^2$), serves as a direct indicator of their dynamic fidelity.

The results are summarized in Figures \ref{fig:hoph50}-\ref{fig:homoclinic_node}, which illustrate the predicted membrane potential $V(t)$ and recovery variable $N(t)$ compared to the ground truth solutions obtained via numerical integration.
In regimes where no bifurcation is present—such as for $I=50\mu A/cm^2$ under the Hopf parameters (Figure \ref{fig:hoph50})—both models rapidly converge to accurate approximations of the system's dynamics. Even with a relatively small number of training epochs, PINNS and NODE successfully reproduce the stable fixed-point behavior of the model.
\begin{figure}[!htbp]
    \centering
    \begin{minipage}[t]{0.48\textwidth}
        \centering
        \includegraphics[width=\textwidth]{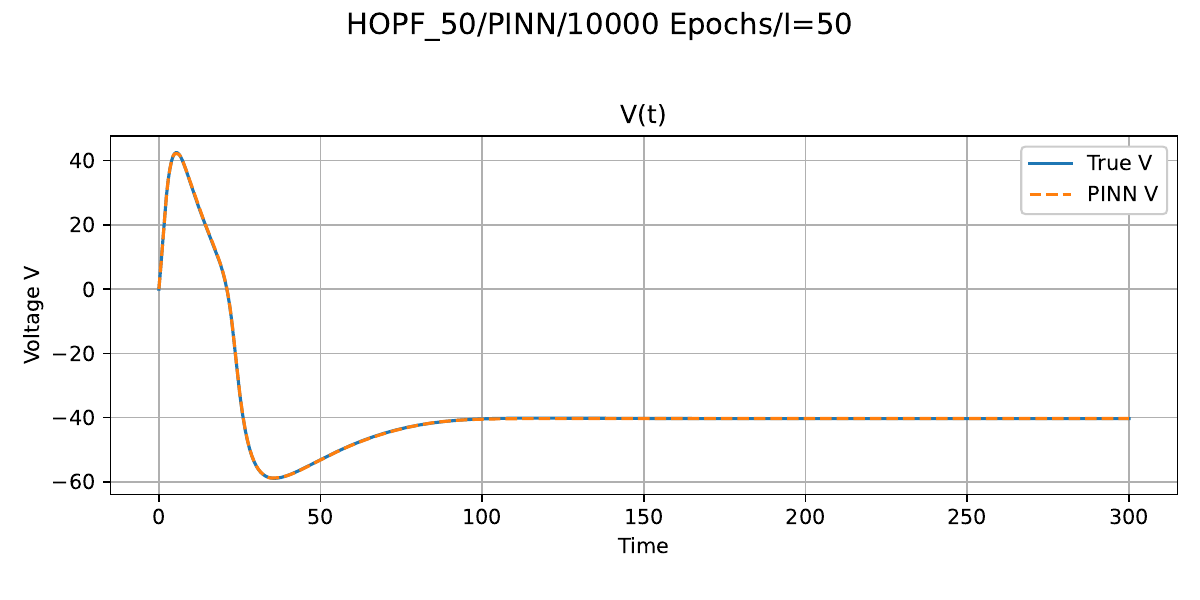}\\
            \end{minipage}
    \hfill
    \begin{minipage}[t]{0.48\textwidth}
        \centering
        \includegraphics[width=\textwidth]{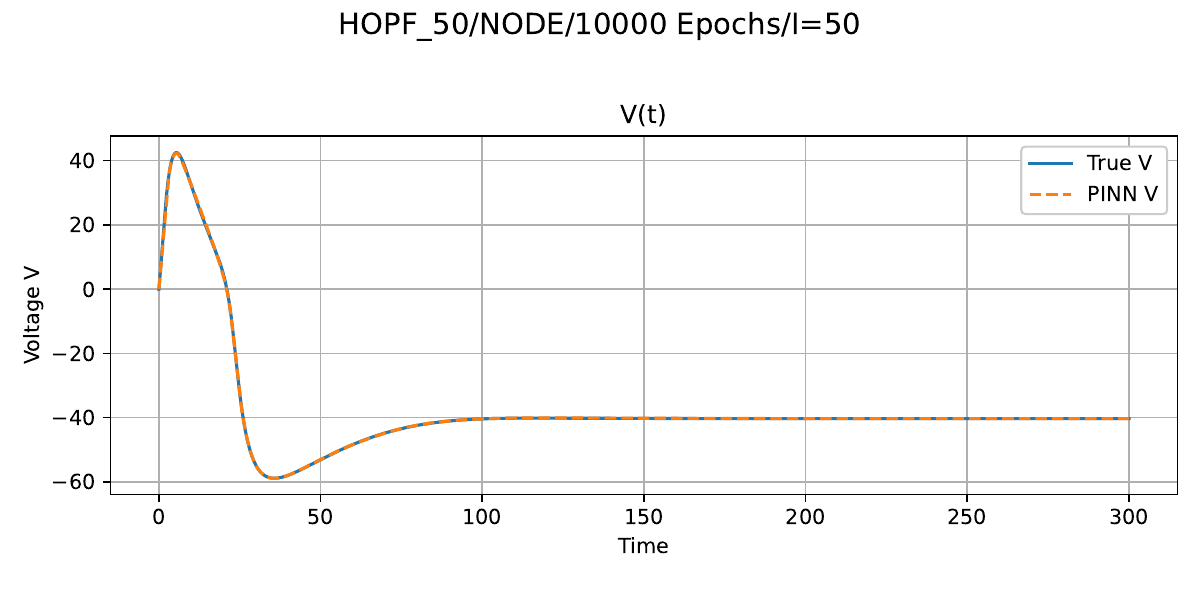}\\
     \end{minipage}
\caption{Predicted vs ground truth voltage trajectories and phase portraits for the Hopf regime with $I_{ext}=50\mu A/cm^2$, a current value that does not induce bifurcation. Both PINN and NODE models successfully reproduce the stable resting dynamics.}
    \label{fig:hoph50}
\end{figure}

In the Hoph bifurcating case of $I_{ext}=90 \mu A/cm^2$, both the PINN and NODE models are able to capture the sustained oscillations characteristic of the Hopf regime. The predictions align well with the true dynamics in both the time and phase space, indicating that—given sufficient training epochs—both models can internalize and generalize the underlying periodic behavior. In Figure \ref{fig:hoph_pinn} demonstrate that in the Hopf regime, the PINN model successfully captures the system's dynamics.
Figure \ref{fig:hoph_node} illustrates the training performance of the NODE model in the bifurcating Hopf dynamics scenario. The loss decreases steadily, suggesting that the model is not only learning the observed data but also incorporating the underlying physical laws embedded in the training objective. This reinforces the effectiveness of the PINN framework in scenarios where the dynamics are governed by known equations.PINN exhibits smooth convergence toward a low training error, effectively capturing the system’s oscillatory dynamics
\begin{figure}[!htbp]
    \centering
    \begin{minipage}[t]{0.48\textwidth}
        \centering
        \includegraphics[width=\textwidth]{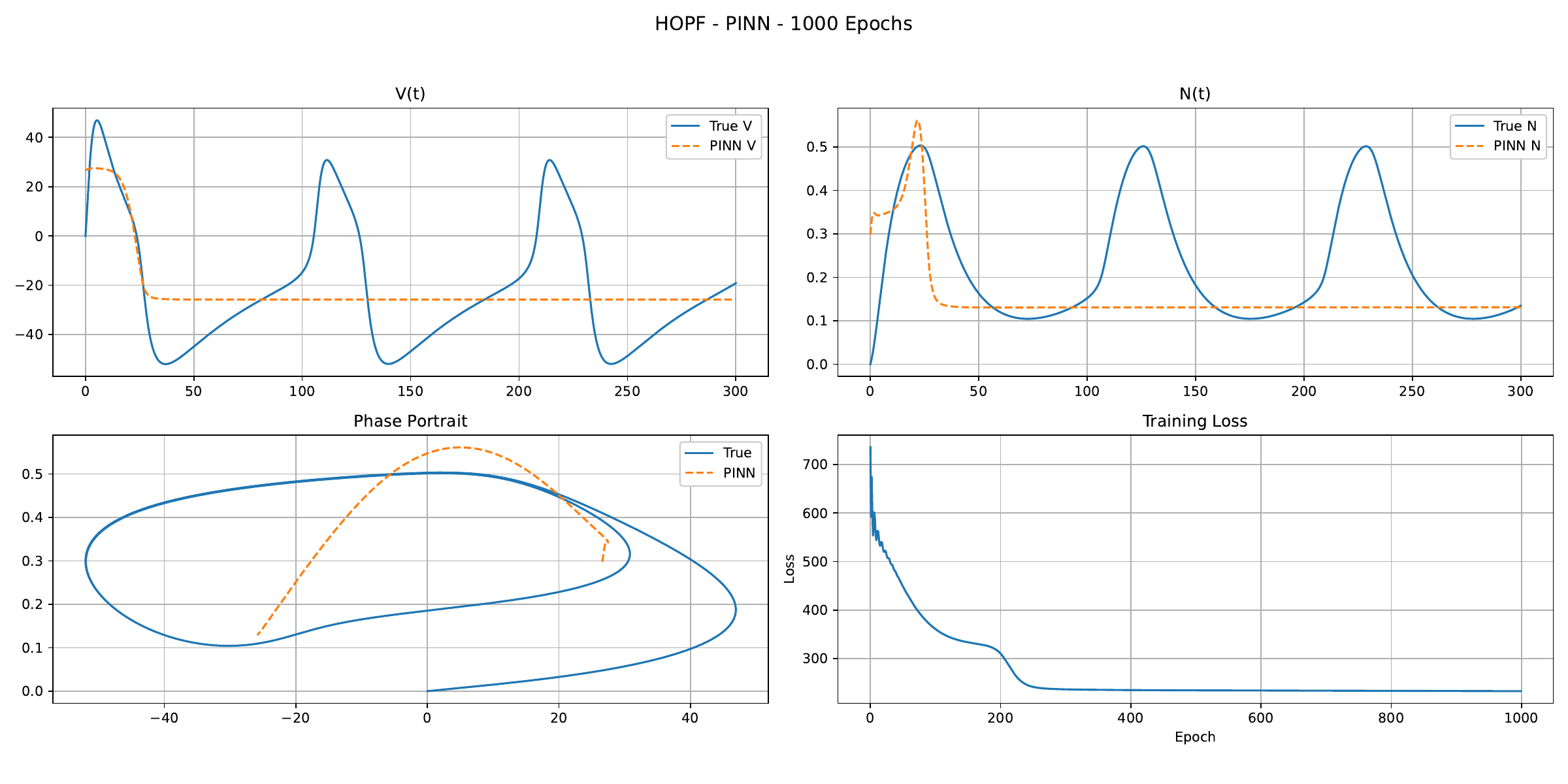}\\
          \end{minipage}
    \hfill
    \begin{minipage}[t]{0.48\textwidth}
        \centering
        \includegraphics[width=\textwidth]{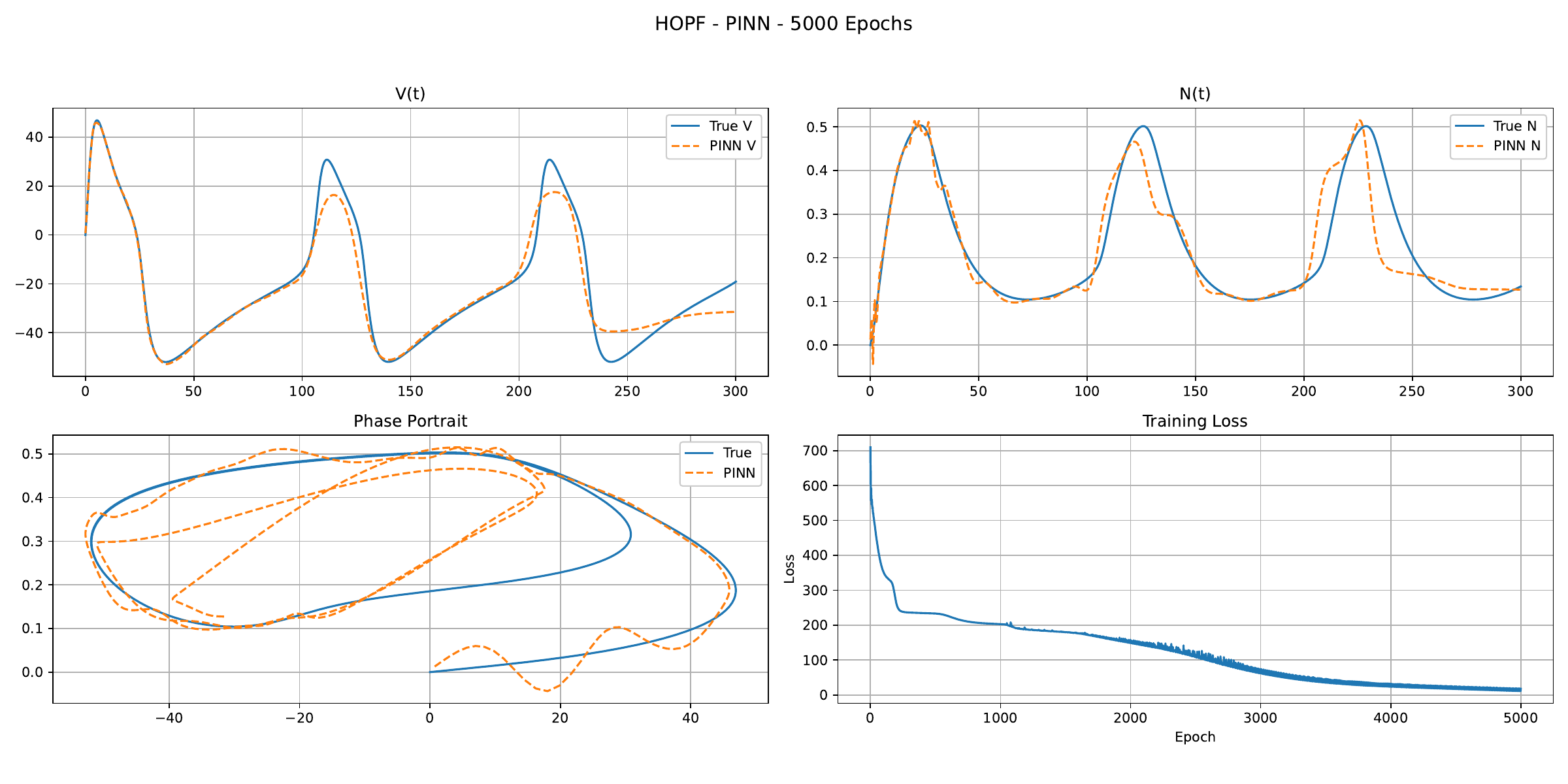}\\
    \end{minipage}
    \begin{minipage}[t]{0.48\textwidth}
        \centering
        \includegraphics[width=\textwidth]{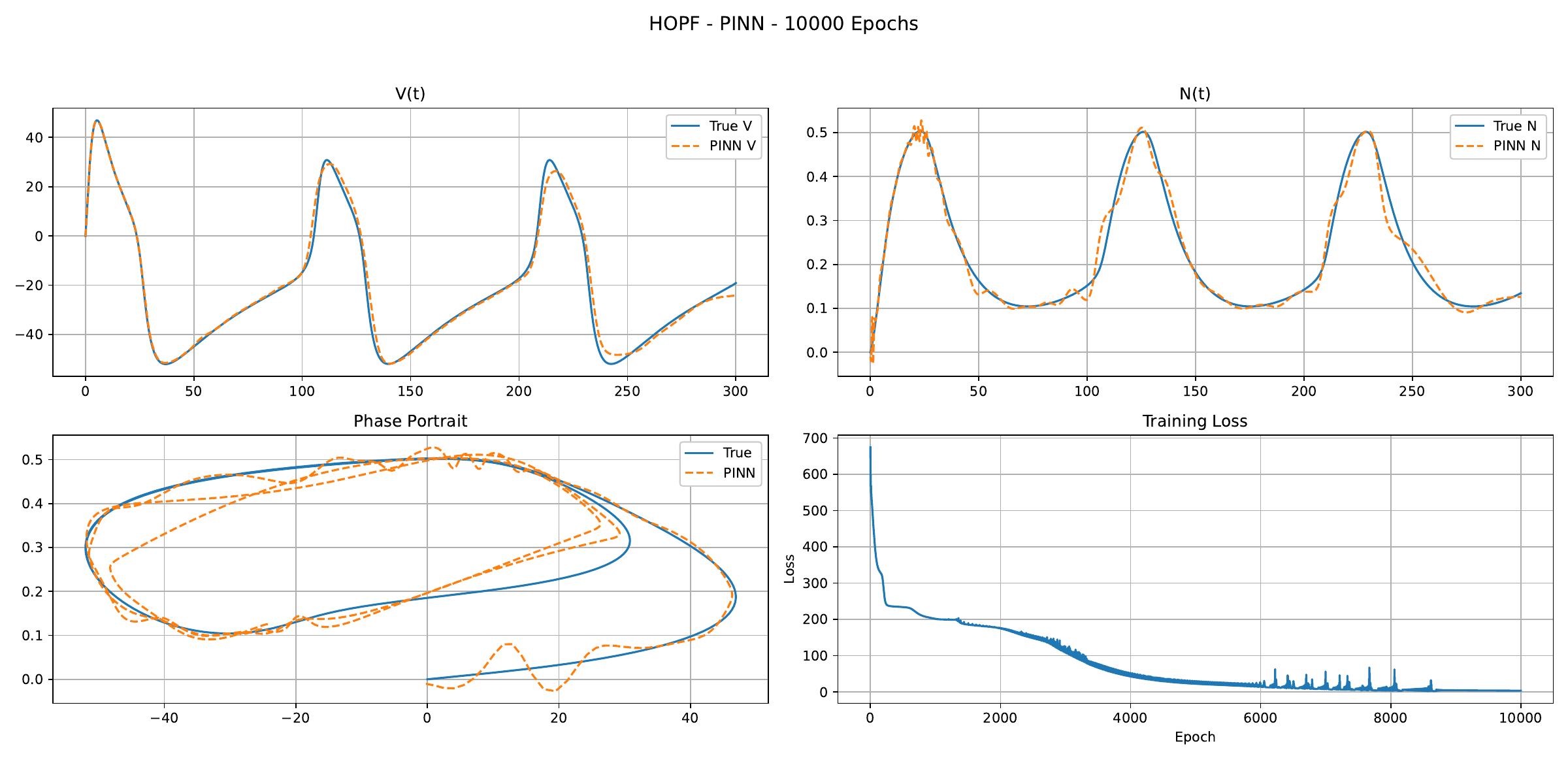}\\
      \end{minipage}
    \hfill
    \begin{minipage}[t]{0.48\textwidth}
        \centering
        \includegraphics[width=\textwidth]{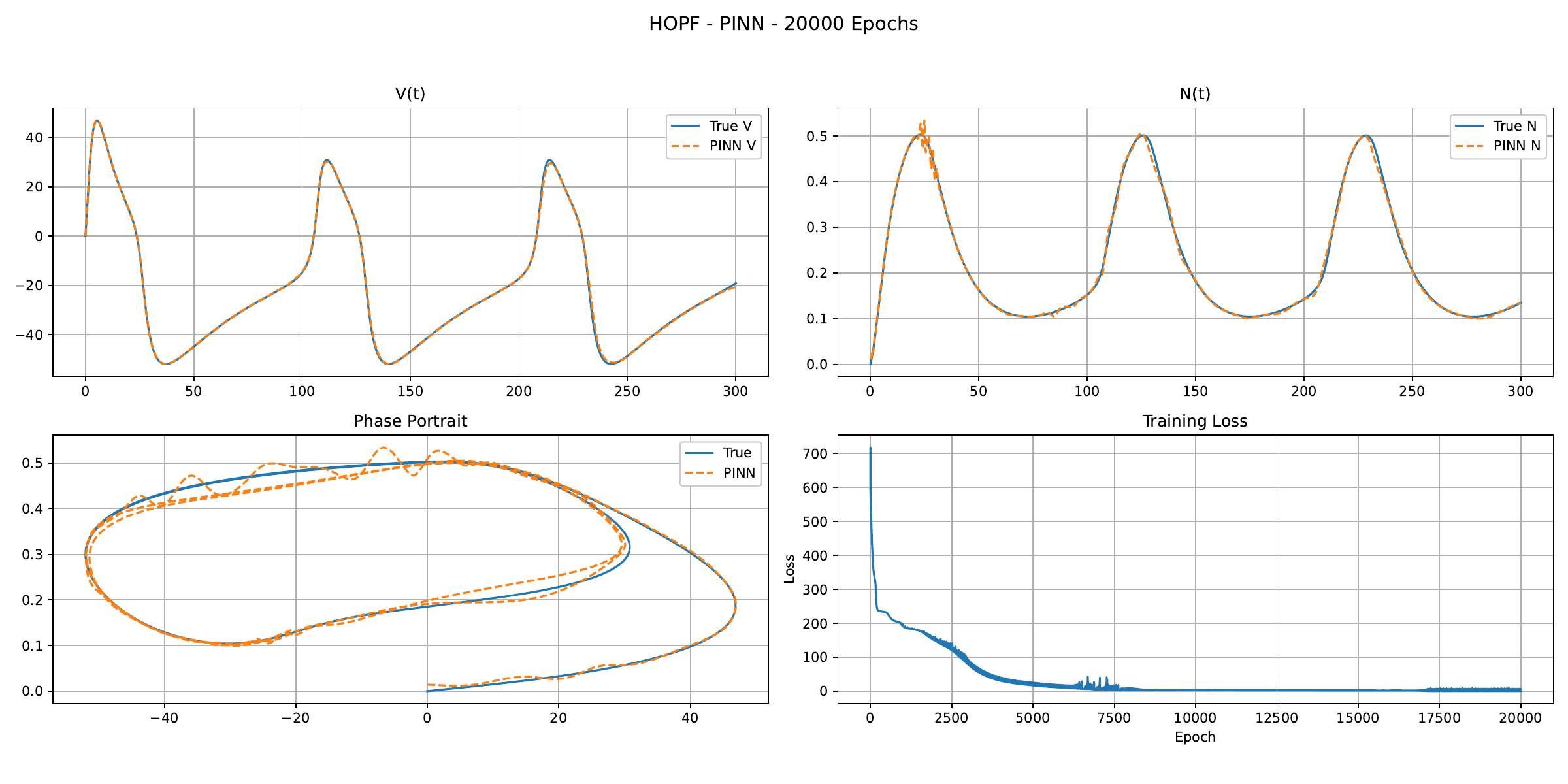}\\
    \end{minipage}
    \caption{PINN predicted vs ground truth trajectories under the Hopf regime for with $I_{ext}=90\mu A/cm^2$ across different training epochs.}
    \label{fig:hoph_pinn}
\end{figure}

\begin{figure}[!htbp]
    \centering
    \begin{minipage}[t]{0.48\textwidth}
        \centering
        \includegraphics[width=\textwidth]{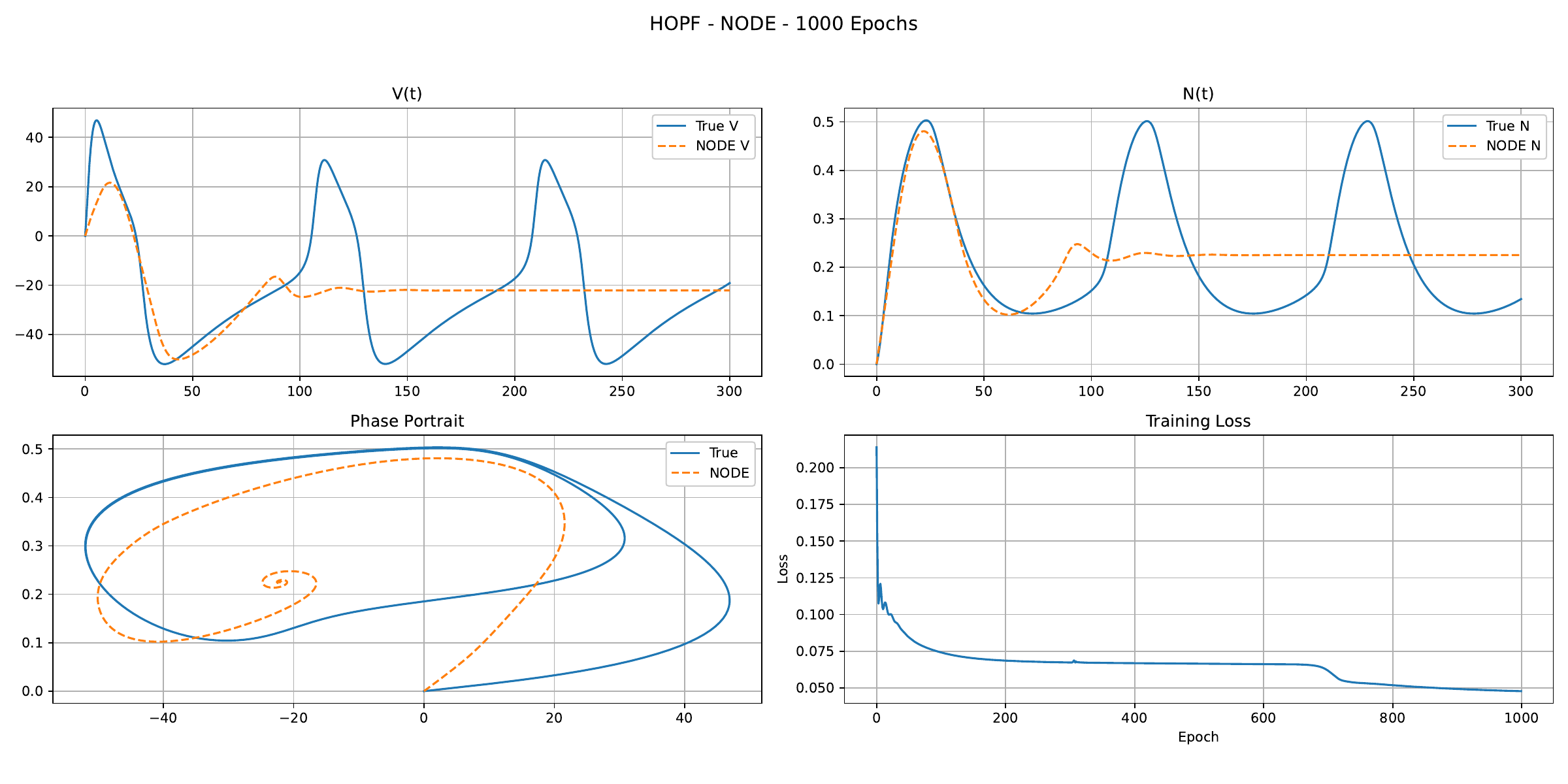}\\
          \end{minipage}
    \hfill
    \begin{minipage}[t]{0.48\textwidth}
        \centering
        \includegraphics[width=\textwidth]{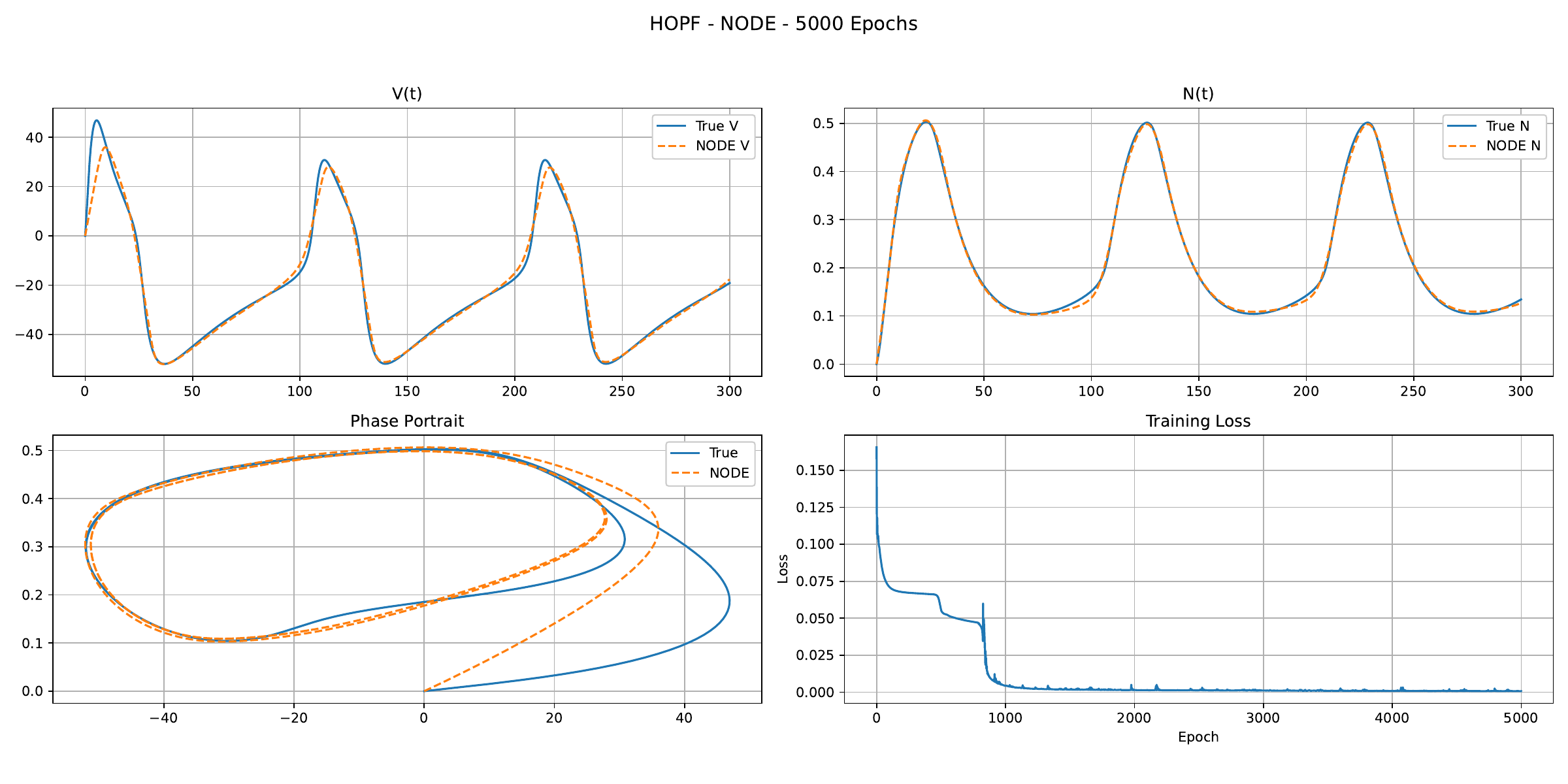}\\
    \end{minipage}
    \begin{minipage}[t]{0.48\textwidth}
        \centering
        \includegraphics[width=\textwidth]{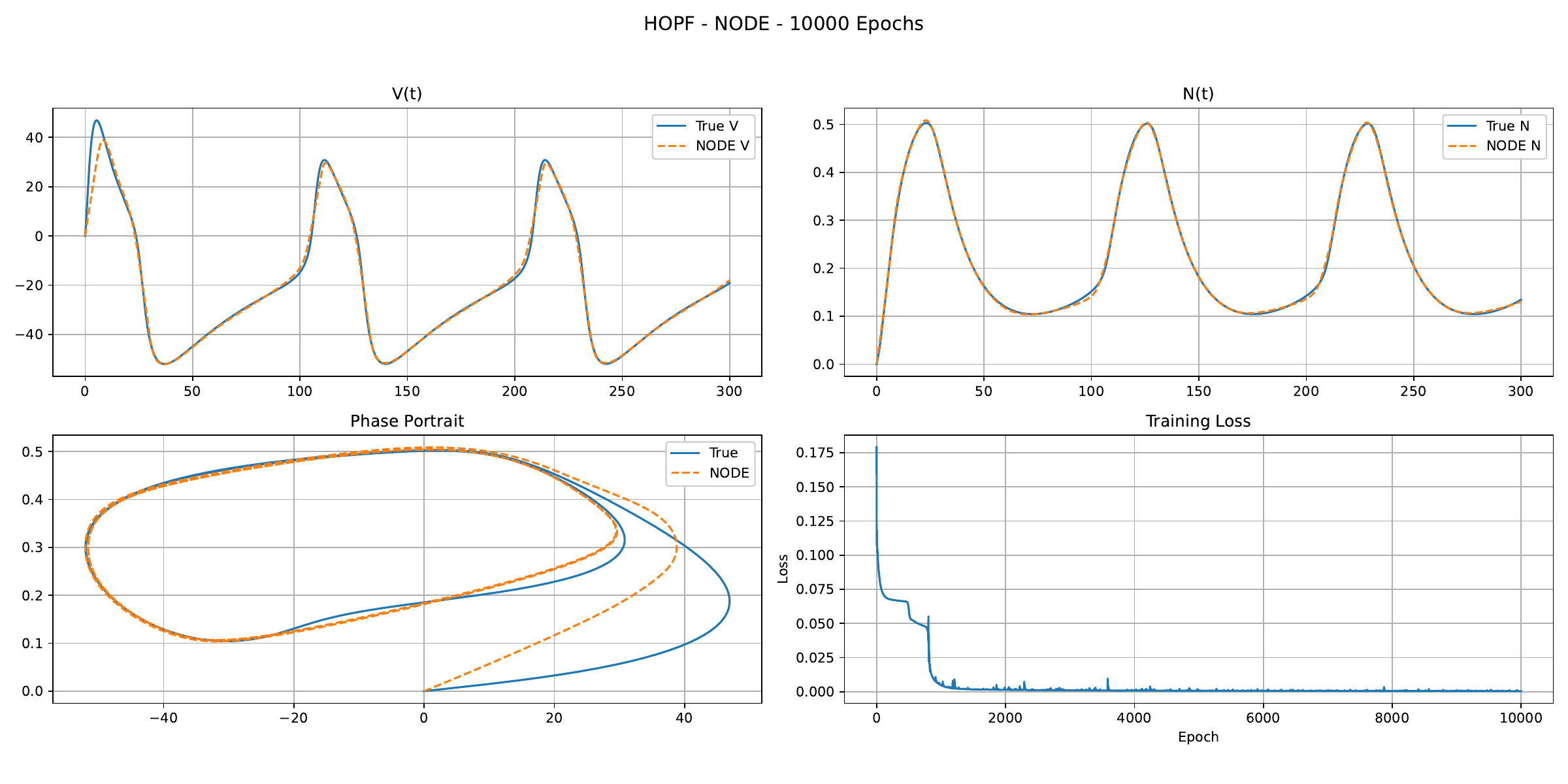}\\
      \end{minipage}
    \hfill
    \begin{minipage}[t]{0.48\textwidth}
        \centering
        \includegraphics[width=\textwidth]{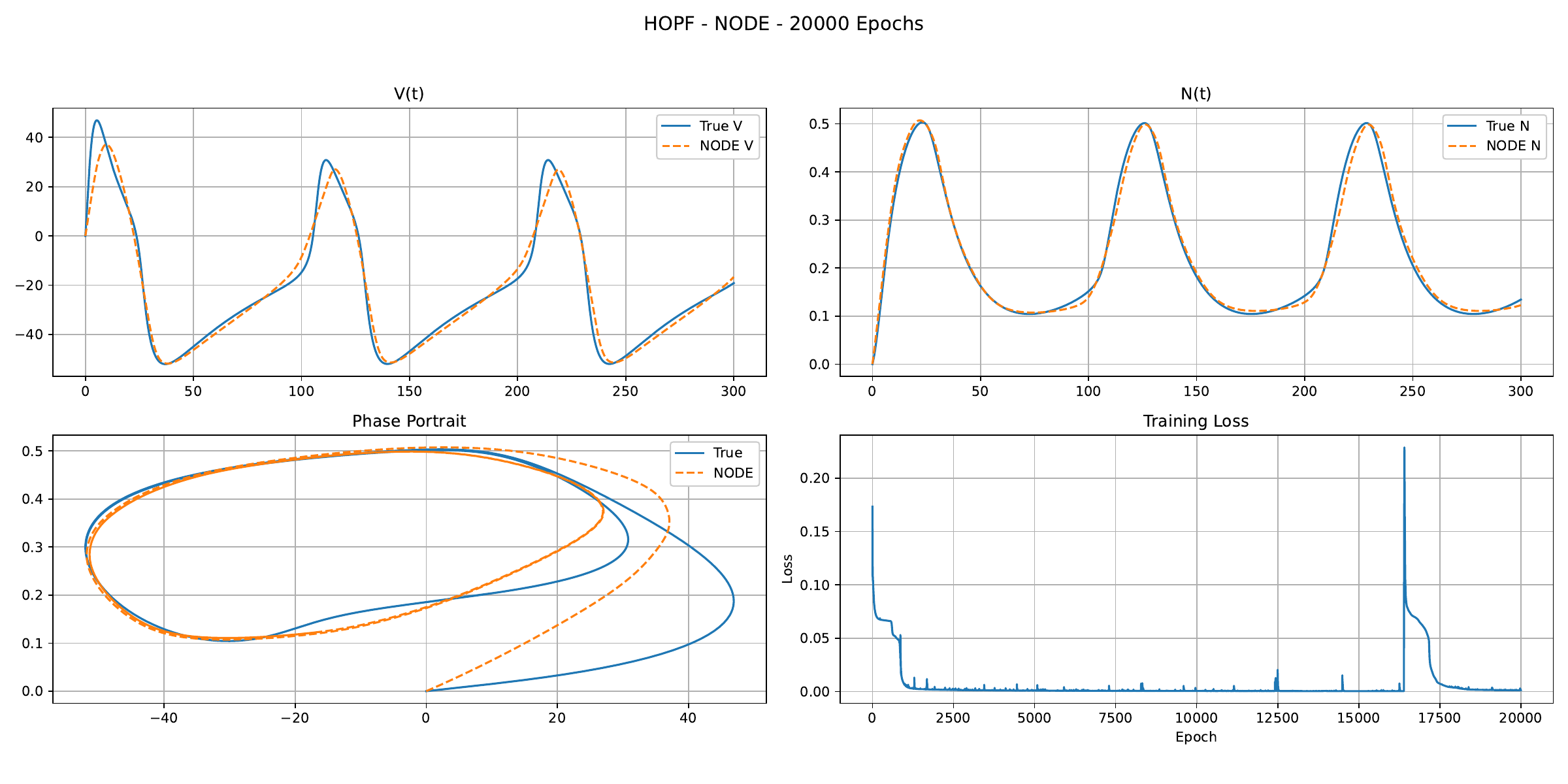}\\
    \end{minipage}

    \caption{NODE predicted vs ground truth trajectories under the Hopf regime for $I_{ext}=90\mu A/cm^2$ across different training epochs.}
    \label{fig:hoph_node}
\end{figure}

A similar behavior is observed in the SNLC regime for input currents $I_{ext}=42 \mu A/cm^2$, which induce bifurcation. As illustrated in Figures \ref{fig:snlc_pinn} and \ref{fig:snlc_node}, both the Neural ODE and PINN architectures successfully reproduce the system’s oscillatory dynamics under these conditions. 
\begin{figure}[!htbp]
    \centering
    \begin{minipage}[t]{0.48\textwidth}
      \centering
      \includegraphics[width=\textwidth]{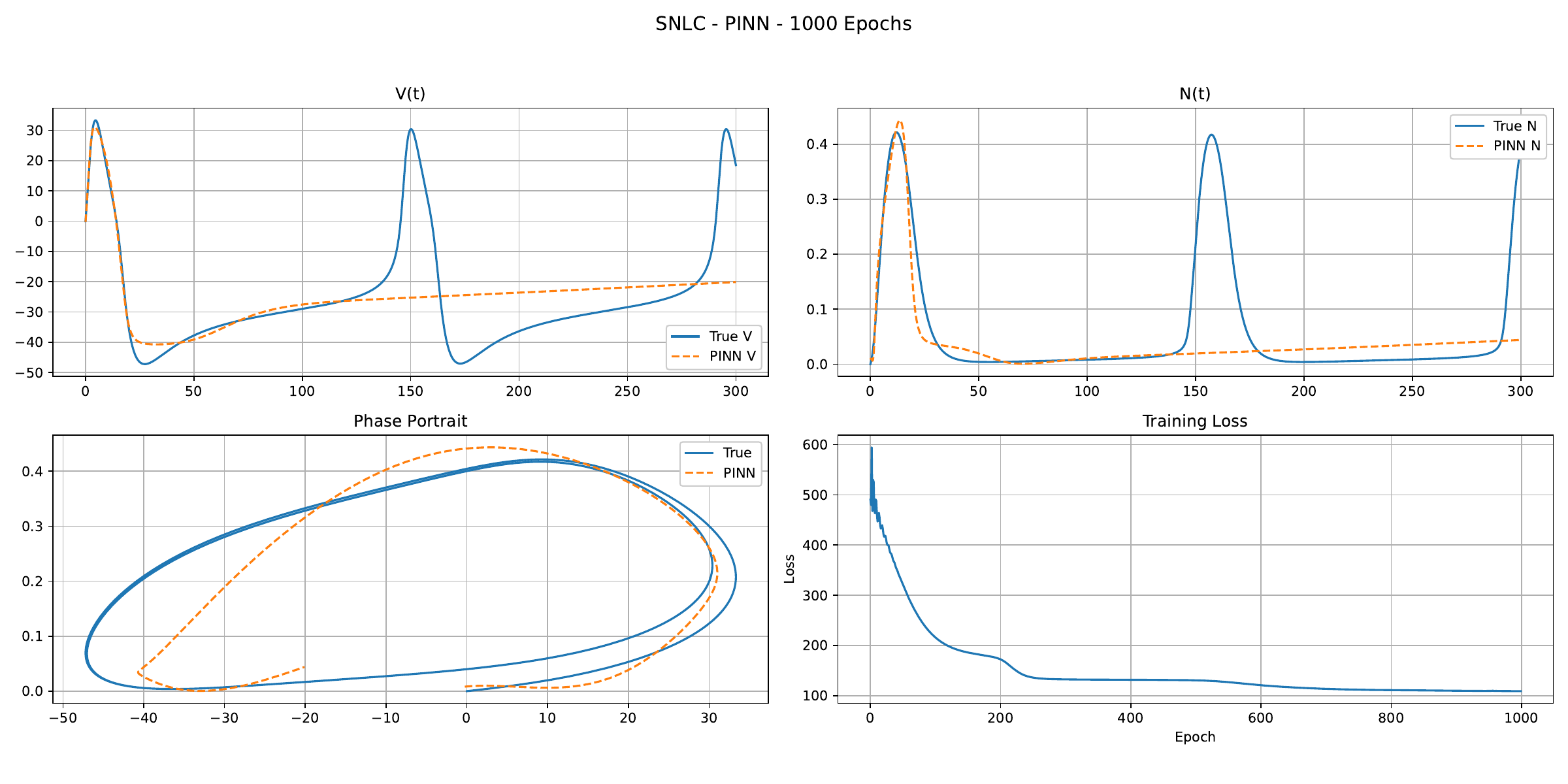}\\
    \end{minipage}
    \hfill
    \begin{minipage}[t]{0.48\textwidth}
        \centering
        \includegraphics[width=\textwidth]{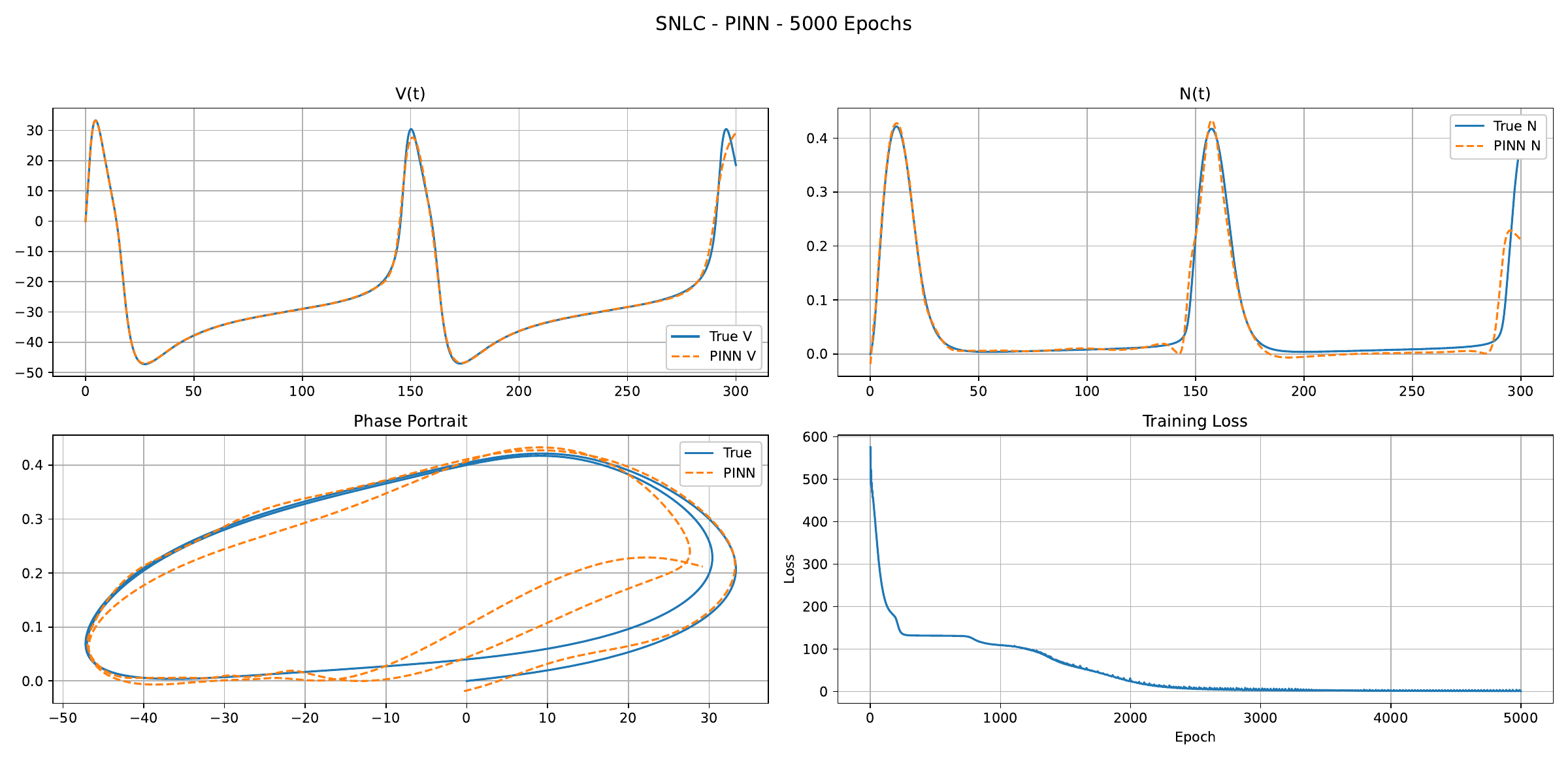}\\
    \end{minipage}
    \begin{minipage}[t]{0.48\textwidth}
        \centering
        \includegraphics[width=\textwidth]{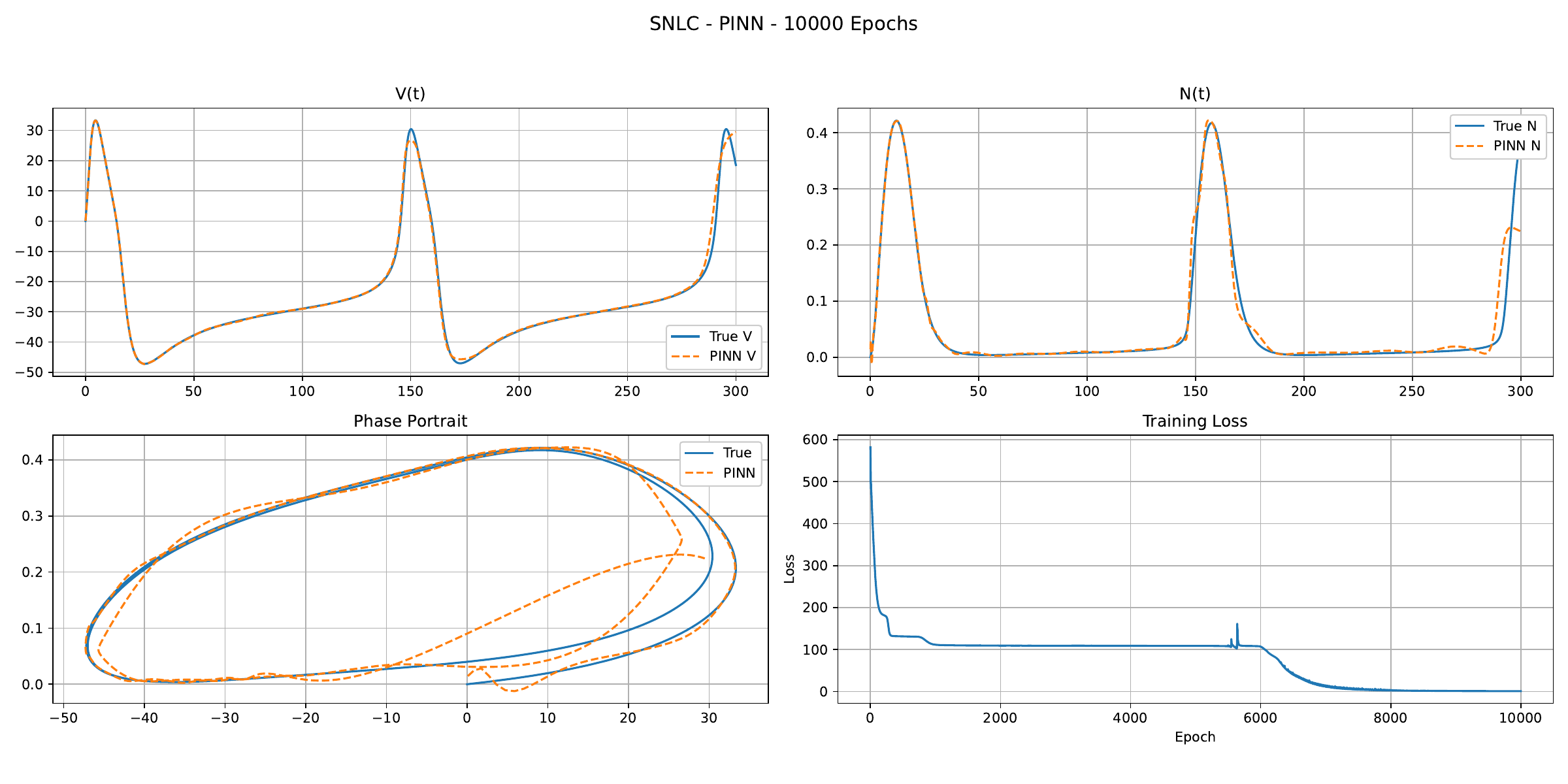}\\
    \end{minipage}
    \hfill
    \begin{minipage}[t]{0.48\textwidth}
        \centering
        \includegraphics[width=\textwidth]{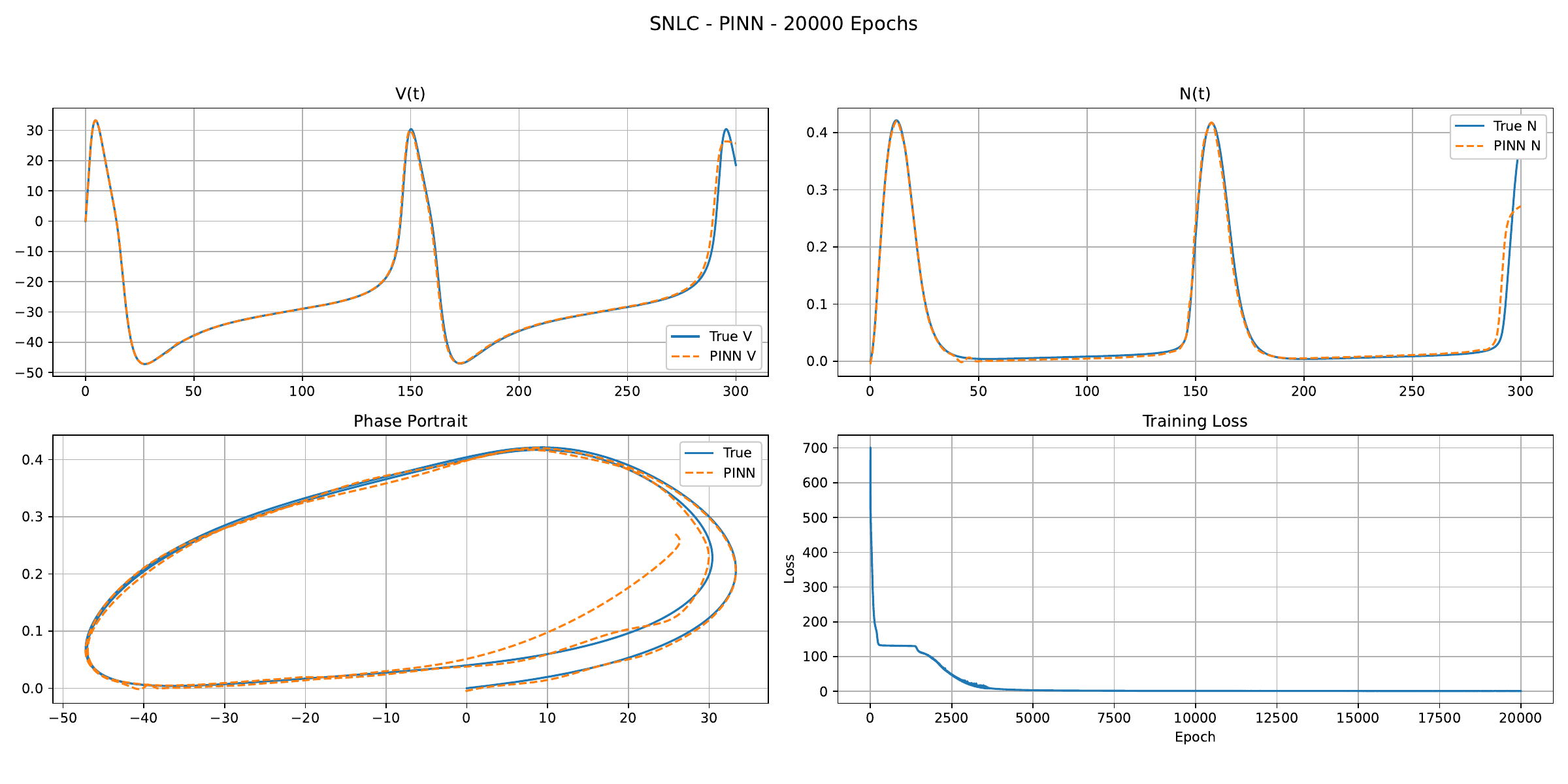}\\
    \end{minipage}
    \caption{PINN model under the SNLC regime with $I_{ext}=42\mu A/cm^2$ .}
    \label{fig:snlc_pinn}
\end{figure}
\begin{figure}[!htbp]
    \centering
     \begin{minipage}[t]{0.48\textwidth}
        \centering
        \includegraphics[width=\textwidth]{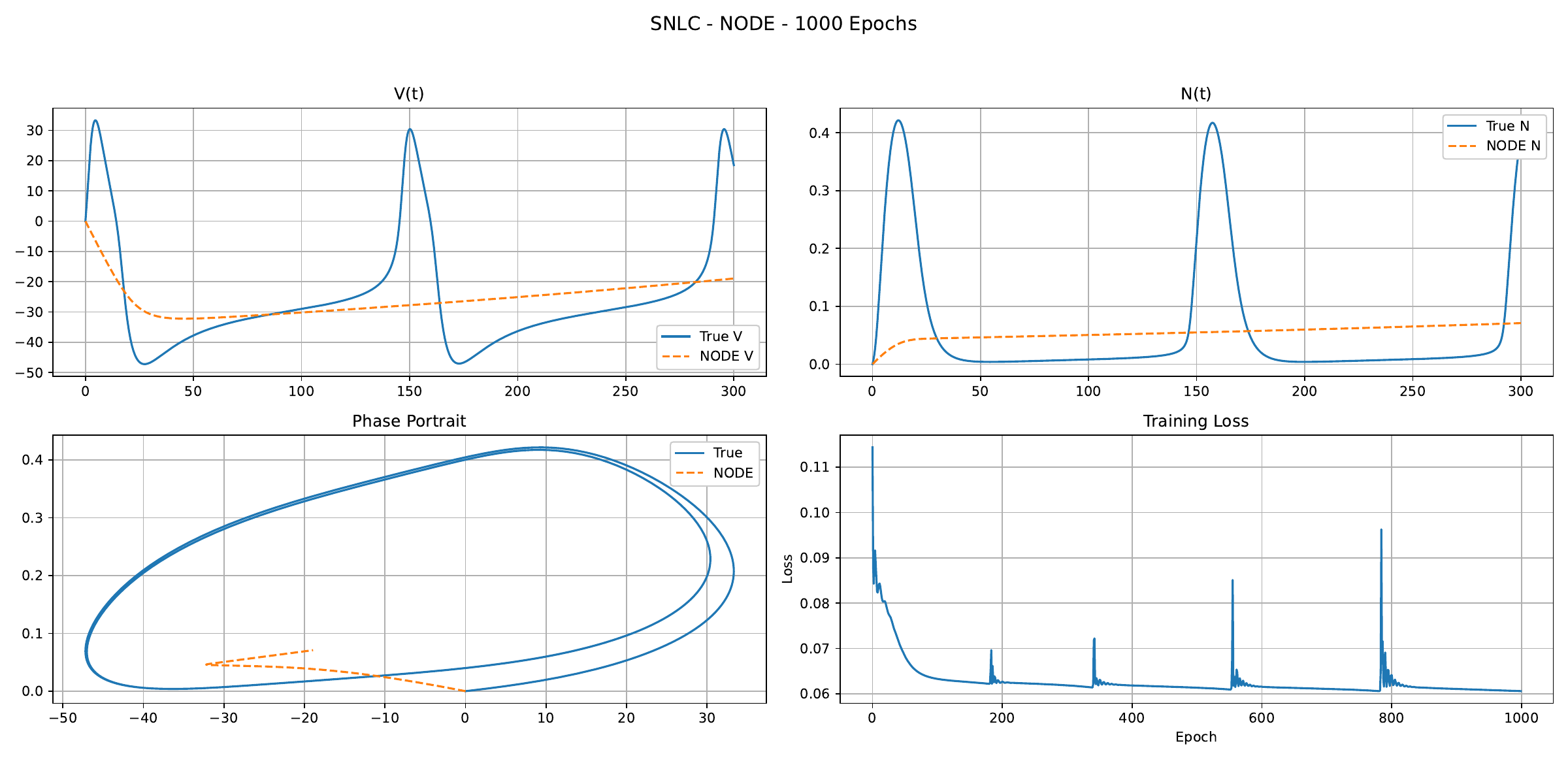}\\
     \end{minipage}
     \hfill
     \begin{minipage}[t]{0.48\textwidth}
        \centering
        \includegraphics[width=\textwidth]{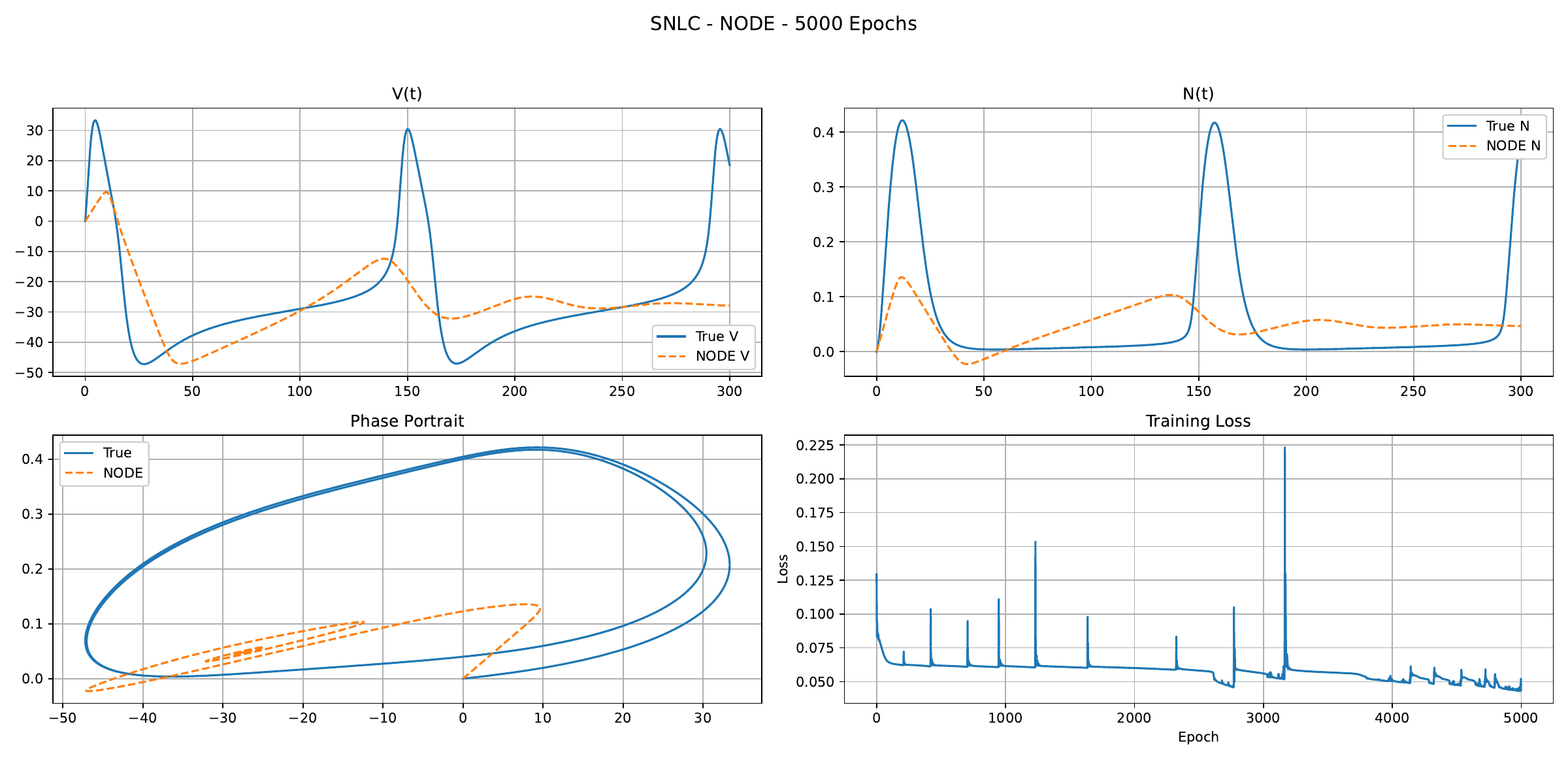}\\
     \end{minipage}
    \begin{minipage}[t]{0.48\textwidth}
        \centering
        \includegraphics[width=\textwidth]{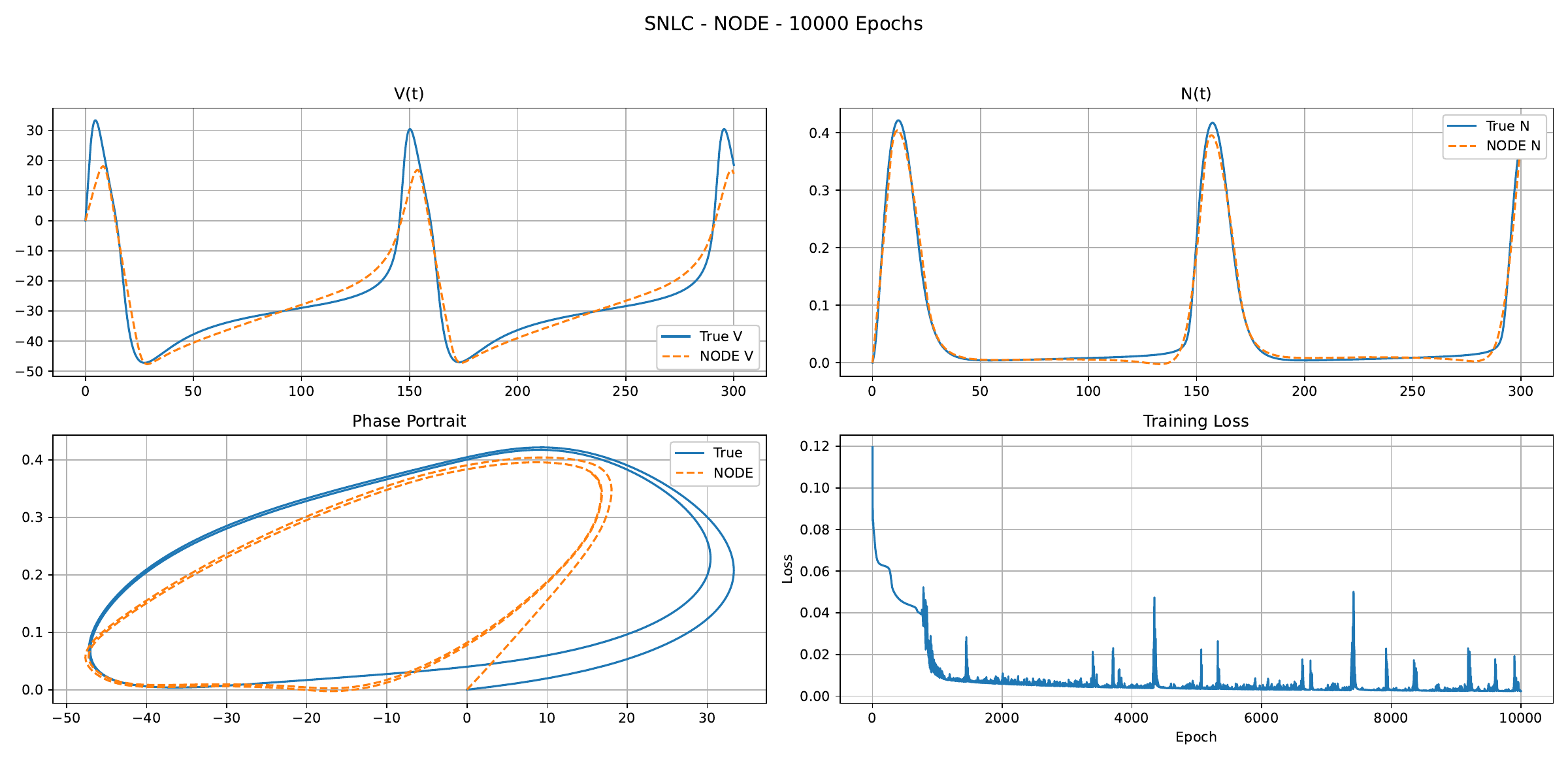}\\
    \end{minipage}
    \hfill
    \begin{minipage}[t]{0.48\textwidth}
        \centering
        \includegraphics[width=\textwidth]{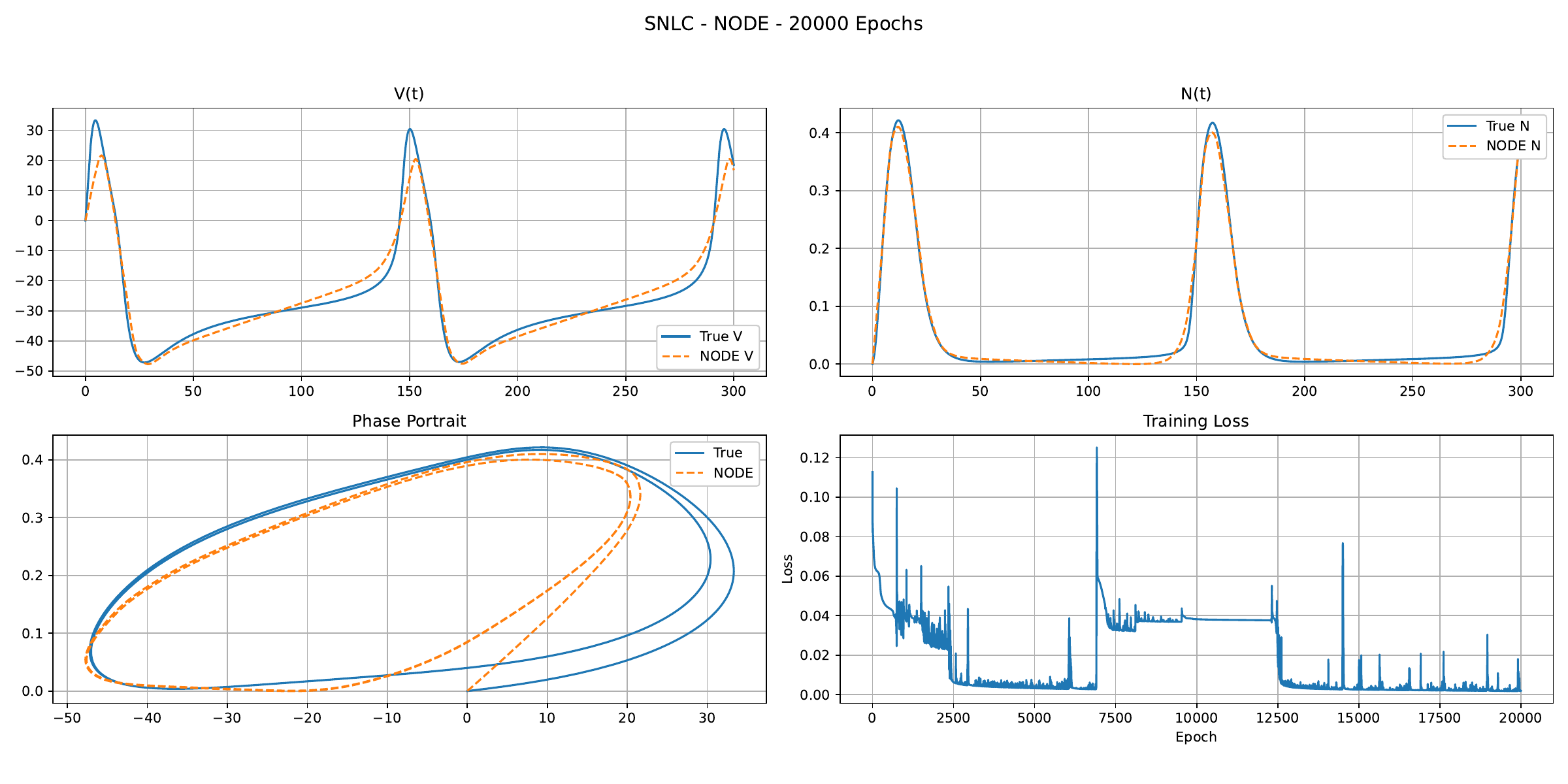}\\
    \end{minipage}

    \caption{NODE model for SNLC regime with $I_{ext}=42\mu A/cm^2$.}
    \label{fig:snlc_node}
\end{figure}

For the Homoclinic regime the PINN model (Figure \ref{fig:homoclinic_pinn}) is capable of accurately approximating the system’s dynamics following $5,000$ training epochs. Each subplot illustrates how increased training improves the model's ability to reproduce the complex bifurcation-induced dynamics, including sharp voltage transitions and slow manifolds.

\begin{figure}[!htbp]
    \centering
    \begin{minipage}[t]{0.48\textwidth}
        \centering
        \includegraphics[width=\textwidth]{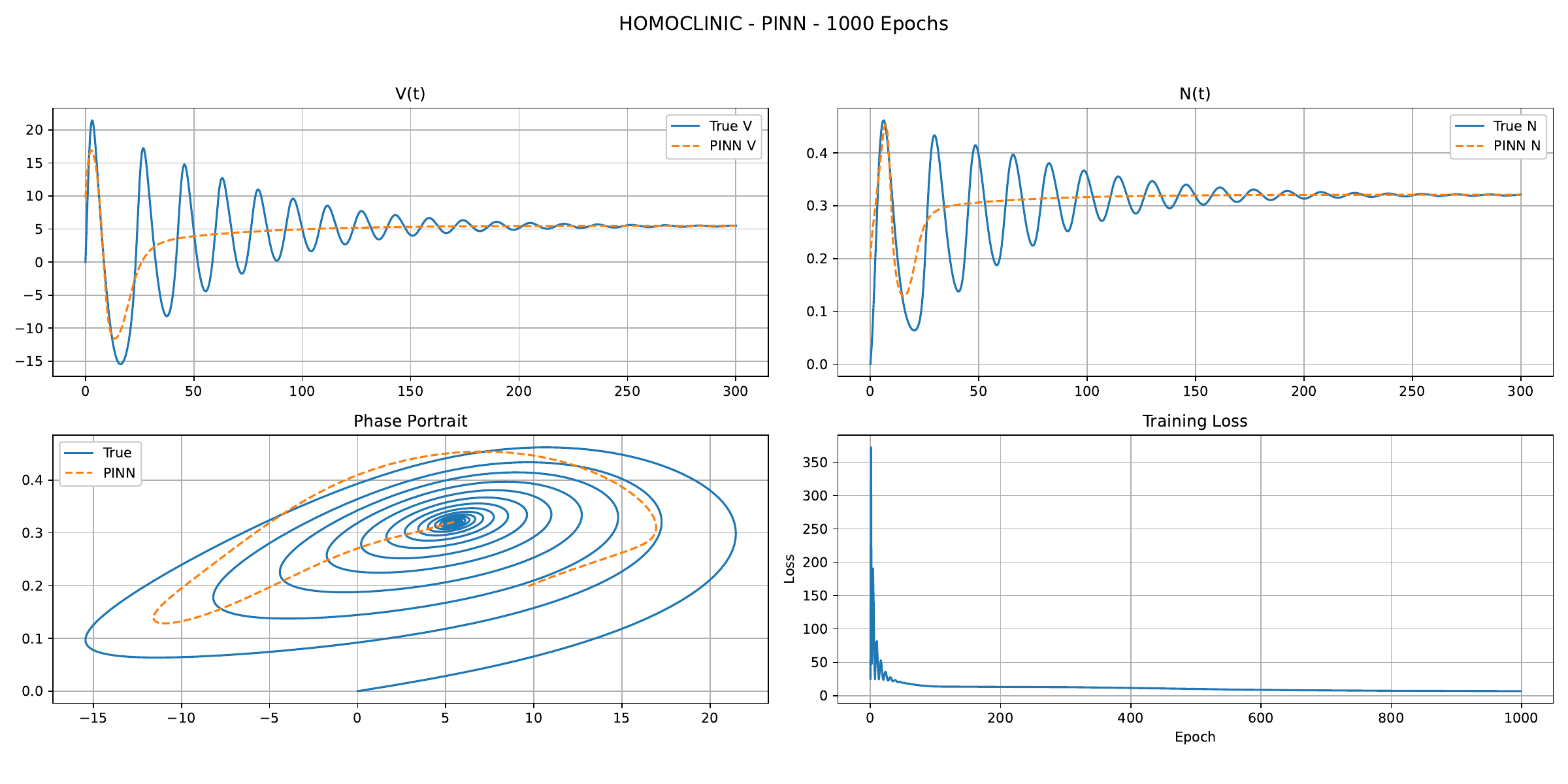}\\
    \end{minipage}
    \hfill
    \begin{minipage}[t]{0.48\textwidth}
        \centering
        \includegraphics[width=\textwidth]{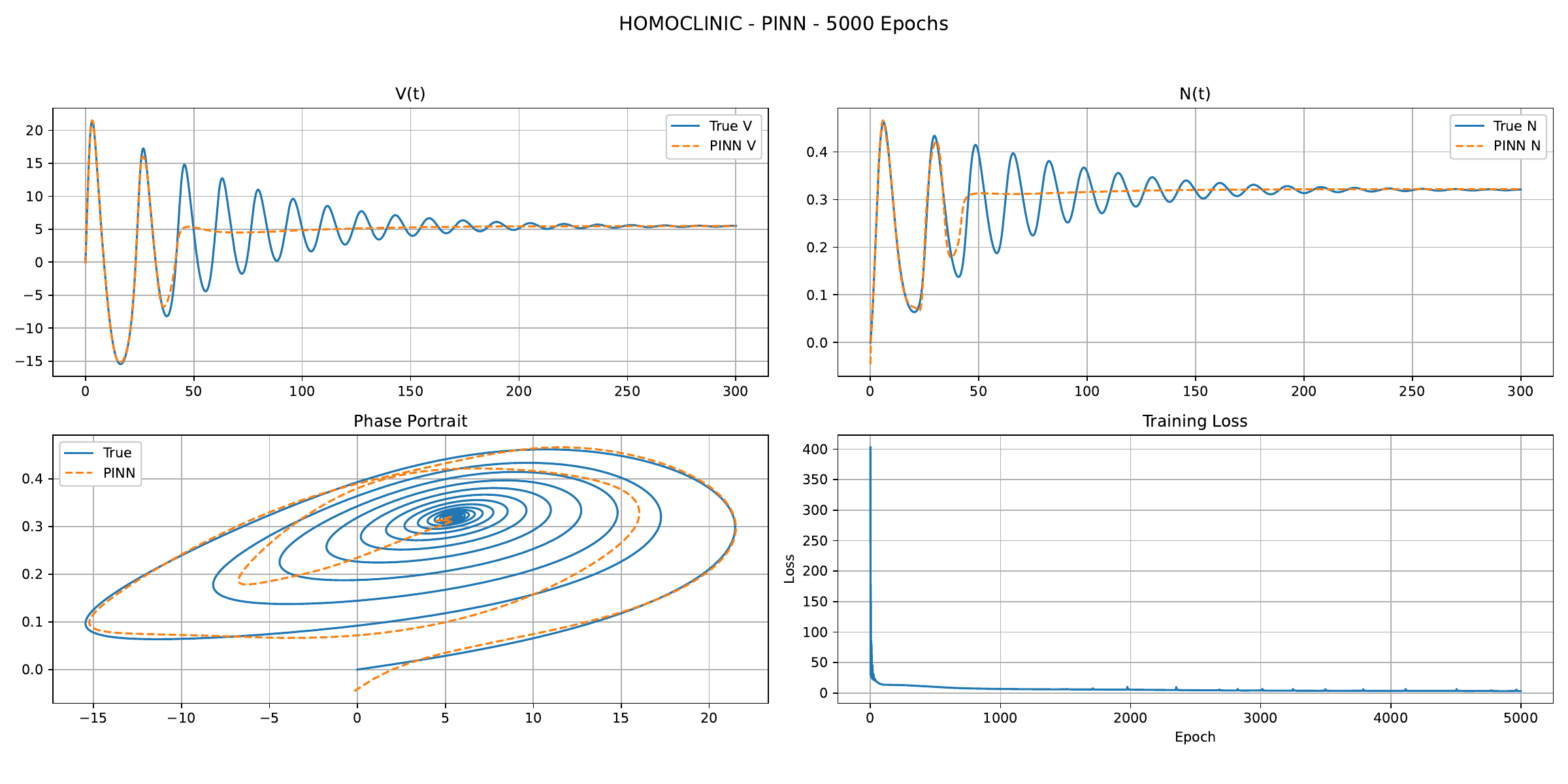}\\
    \end{minipage}
    \begin{minipage}[t]{0.48\textwidth}
        \centering
        \includegraphics[width=\textwidth]{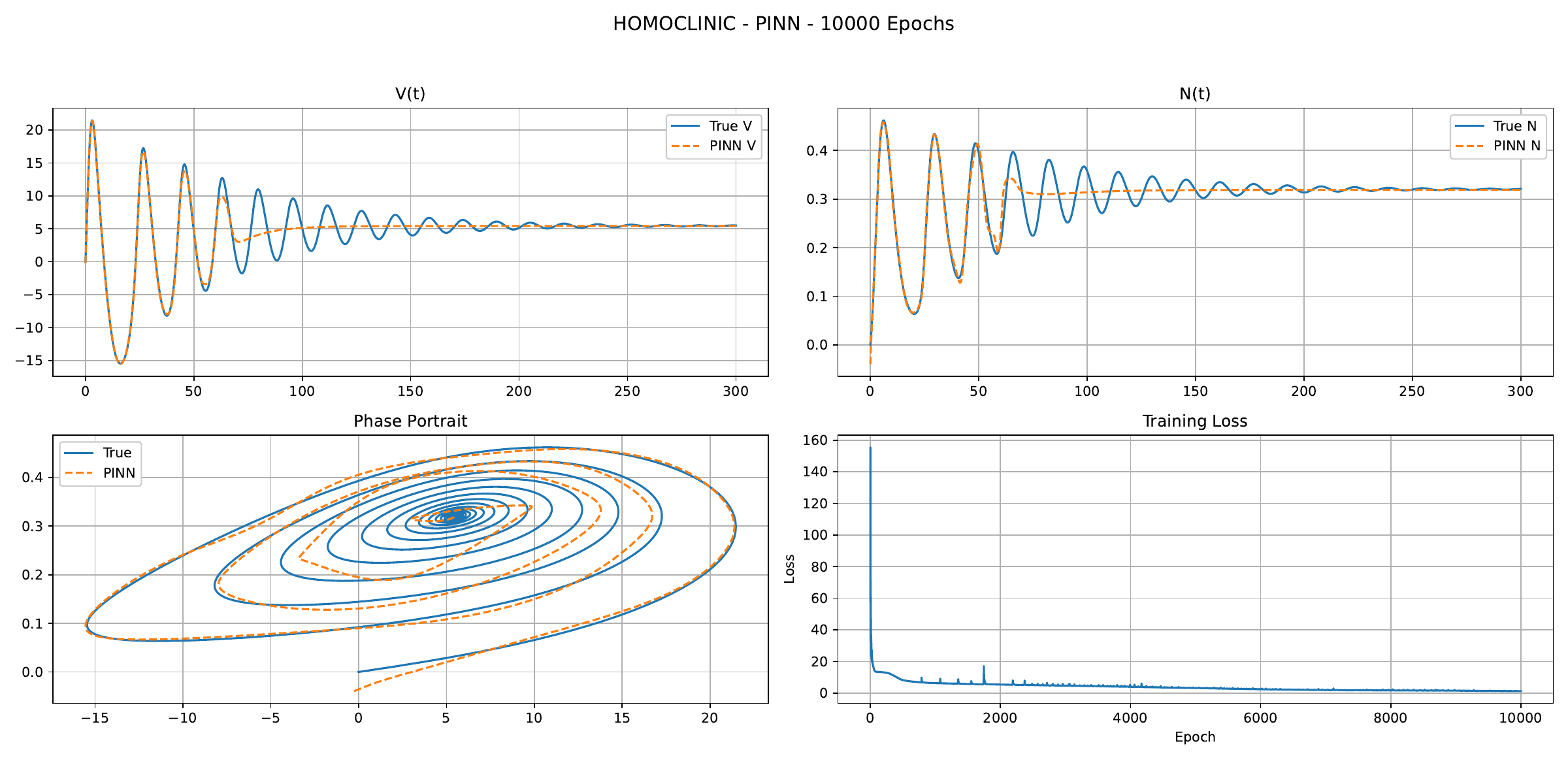}\\
    \end{minipage}
    \hfill
    \begin{minipage}[t]{0.48\textwidth}
        \centering
        \includegraphics[width=\textwidth]{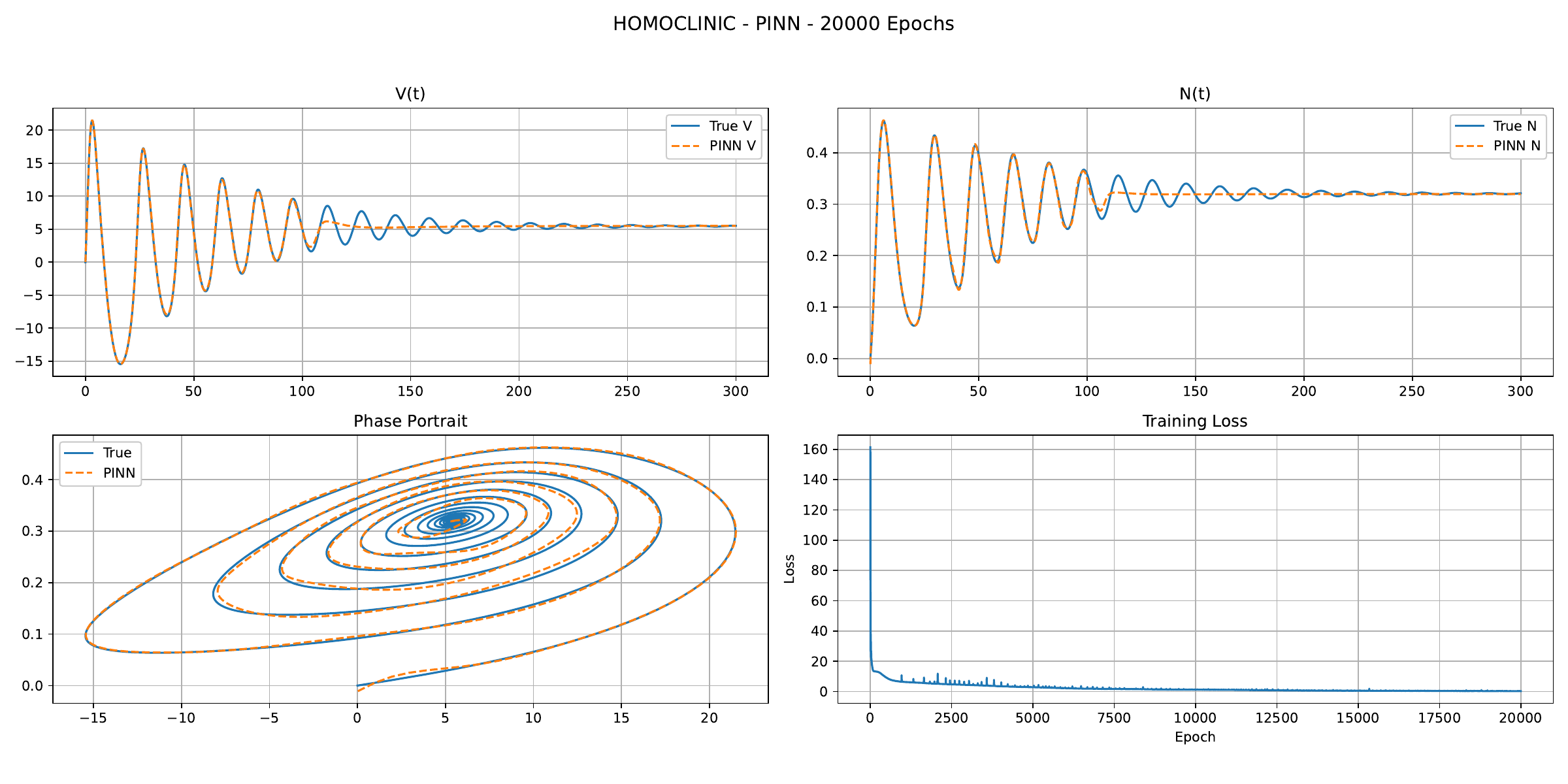}\\
    \end{minipage}

    \caption{PINN-based predictions for the Morris–Lecar model in the homoclinic regime with $I_{ext}=50\mu A/cm^2$ across different training epochs. }
    \label{fig:homoclinic_pinn}
\end{figure}
In contrast,  homoclinic regime $I_{ext}=50\mu A/cm^2$ (Figure \ref{fig:homoclinic_node}), the Neural ODE fails to reproduce the system’s oscillatory behavior, highlighting a limitation of purely data-driven models in this highly sensitive bifurcation scenario.
\begin{figure}[!htbp]
    \centering
    \begin{minipage}[t]{0.48\textwidth}
       \centering
       \includegraphics[width=\textwidth]{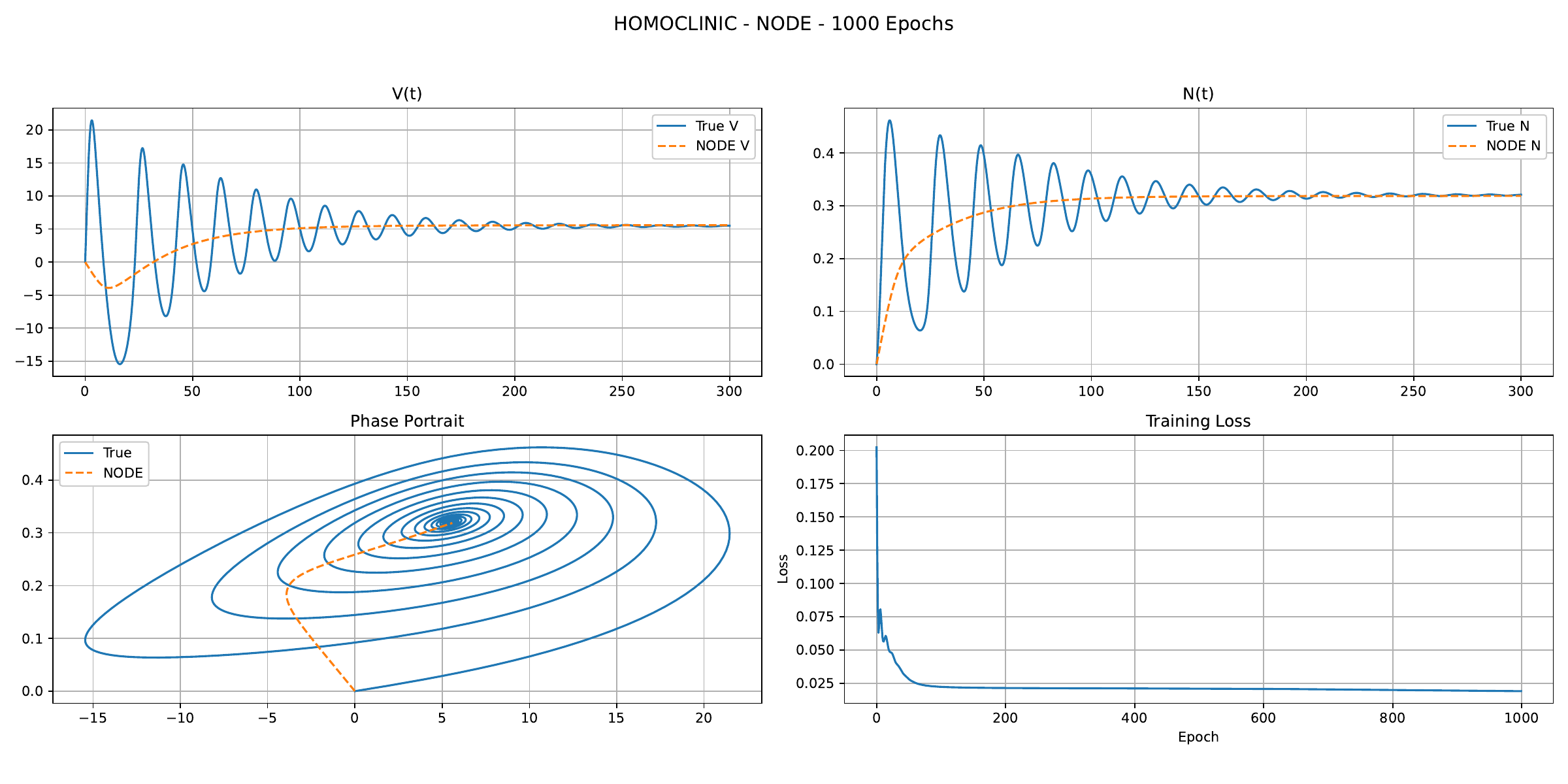}\\
    \end{minipage}
    \hfill
    \begin{minipage}[t]{0.48\textwidth}
        \centering
        \includegraphics[width=\textwidth]{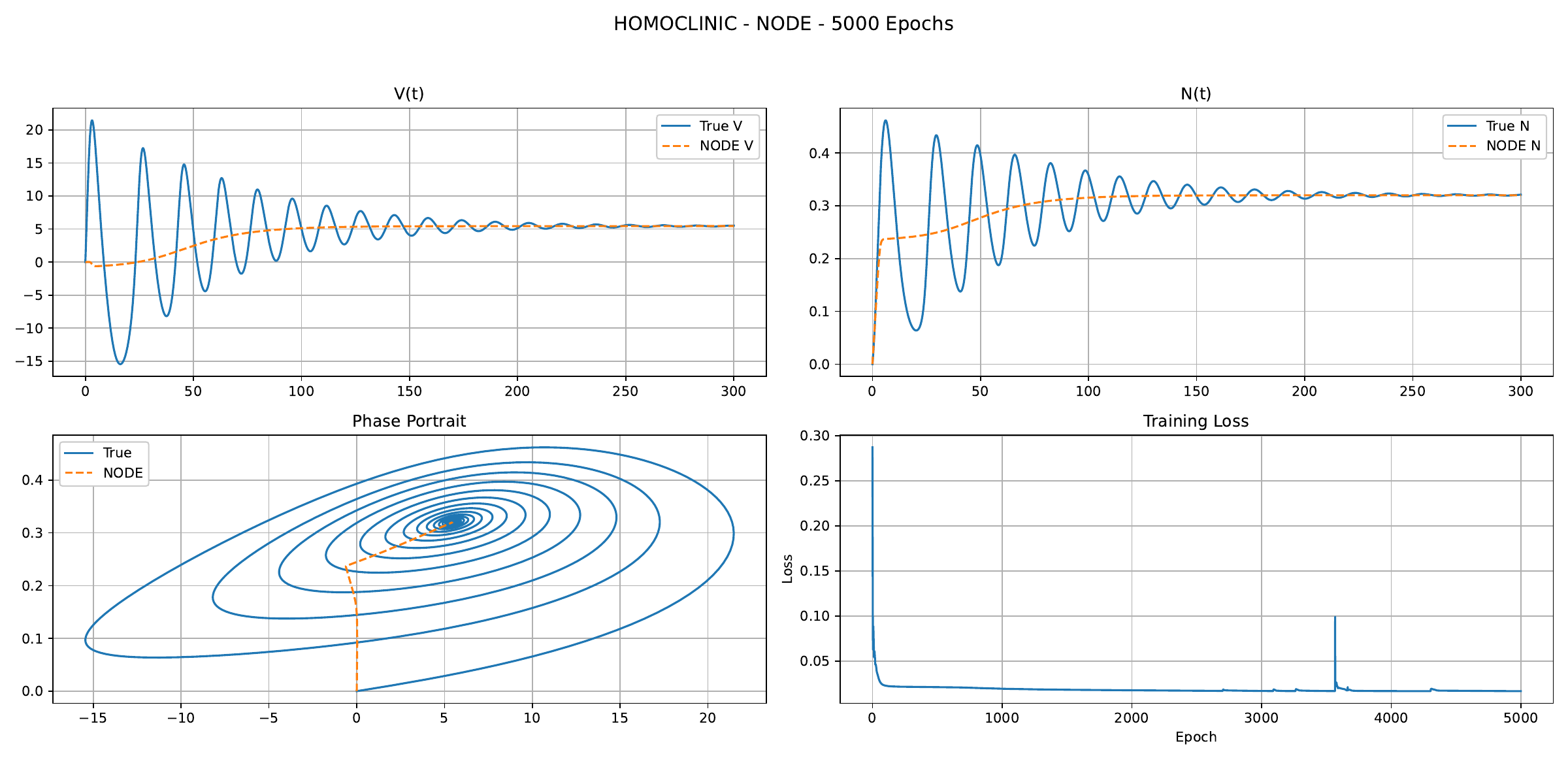}\\
    \end{minipage}
    \begin{minipage}[t]{0.48\textwidth}
        \centering
        \includegraphics[width=\textwidth]{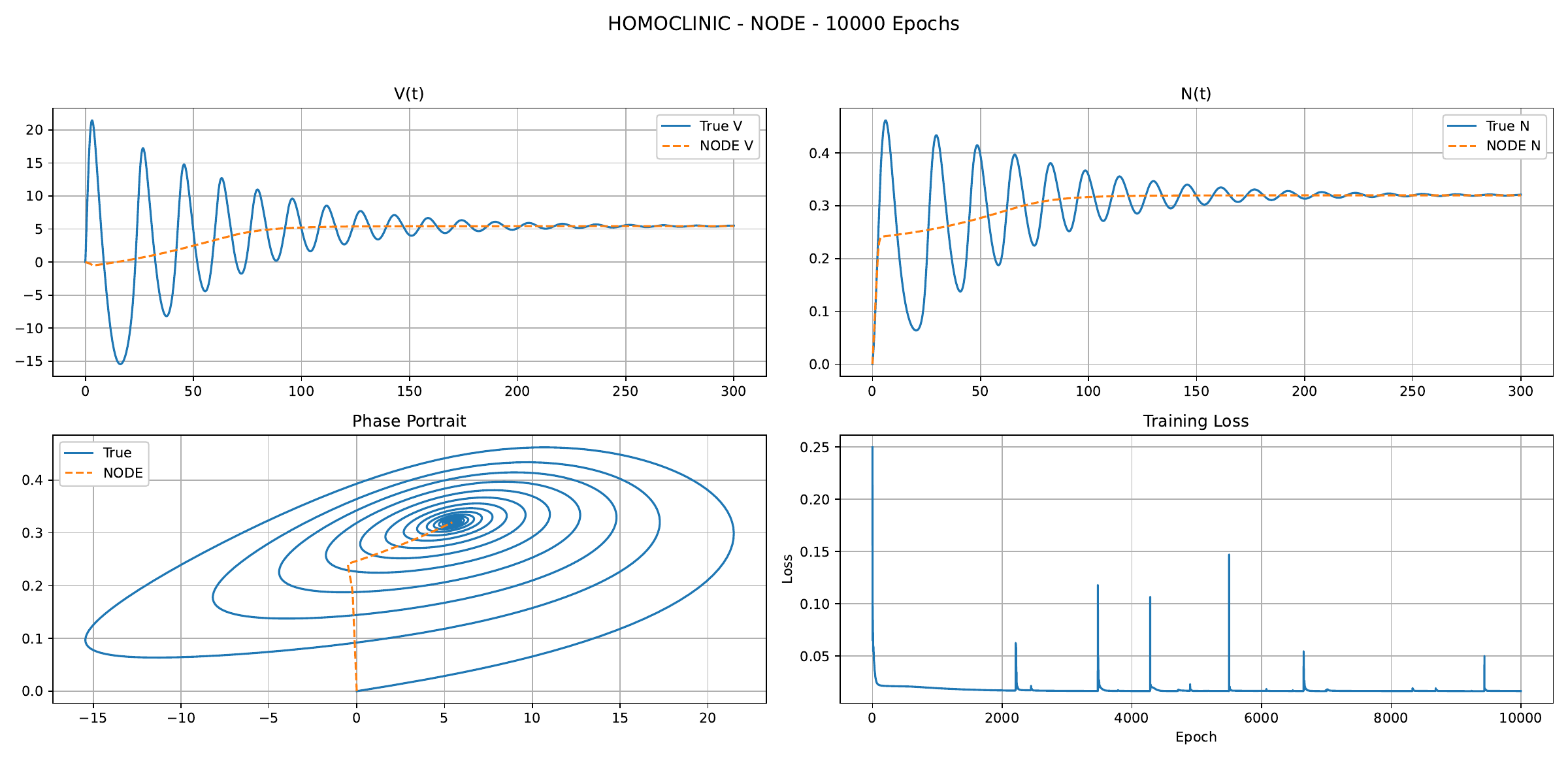}\\
    \end{minipage}
    \hfill
    \begin{minipage}[t]{0.48\textwidth}
        \centering
        \includegraphics[width=\textwidth]{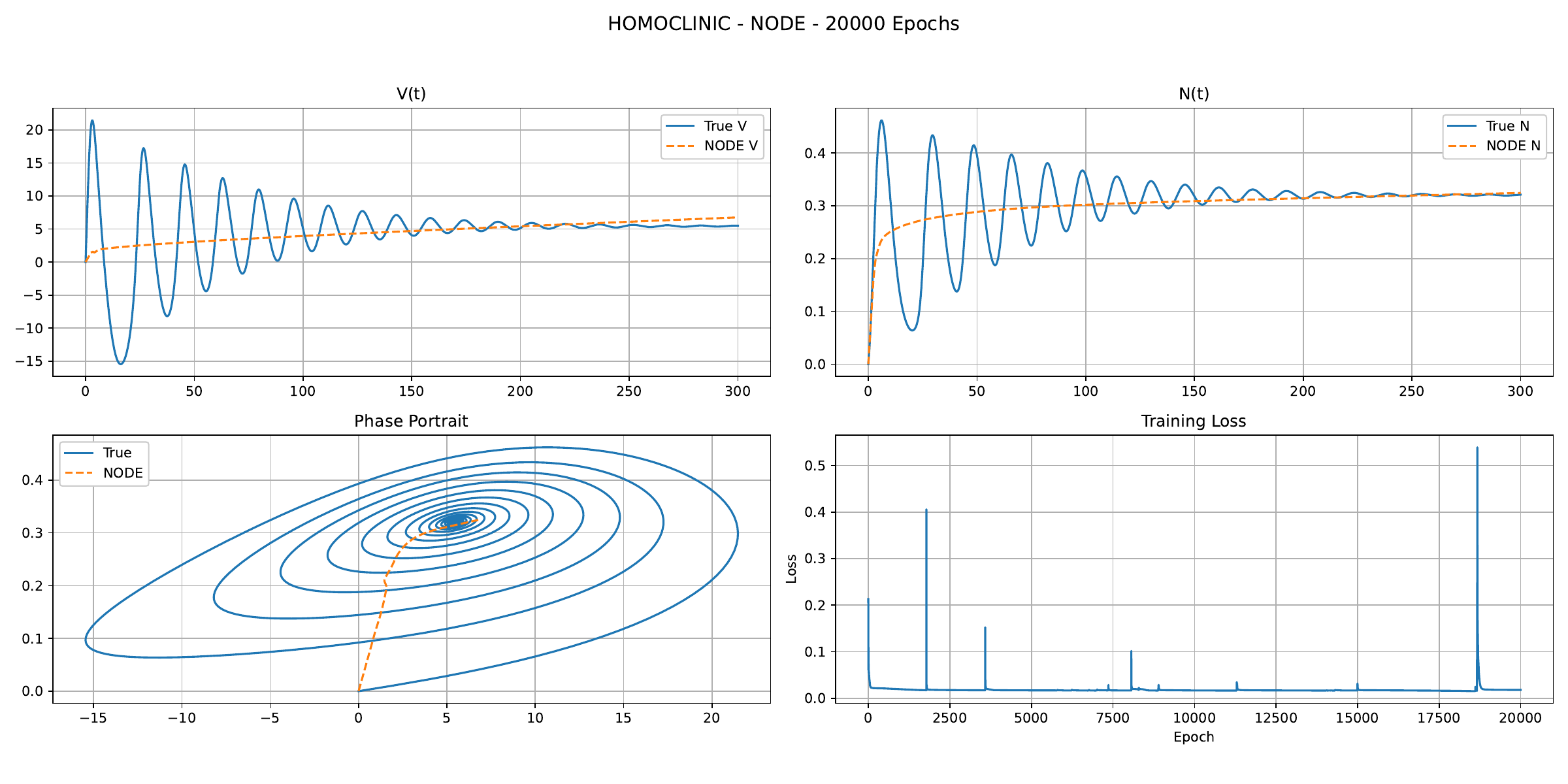}\\
    \end{minipage}
    \caption{NODE model for the homoclinic regime with $I_{ext}=50\mu A/cm^2$.}
    \label{fig:homoclinic_node}
\end{figure}
The results underscore both the challenges and the eventual convergence behavior of the NODE model in capturing homoclinic trajectories. As discussed in Remark~\ref{rem:SiLU}, we employed the SiLU activation function, with the outcome illustrated in Figure~\ref{fig:homoclinicSiLU}. Notably, the NODE successfully reproduces the underlying dynamics of the Morris–Lecar system after $10,000$ training epochs.
\begin{figure}[!htbp]
    \centering
    \includegraphics[width=0.85\textwidth]{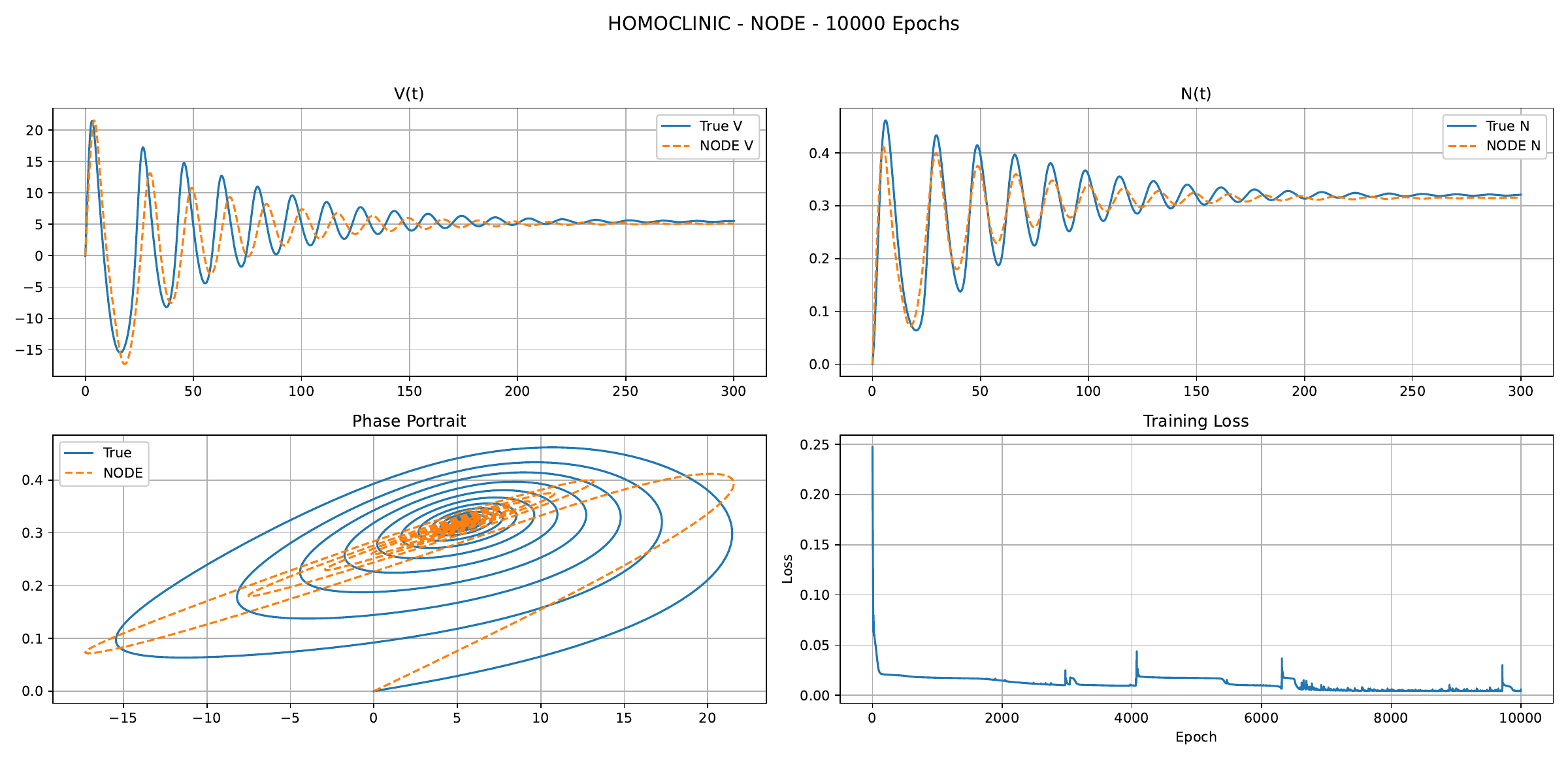}
    \caption{NODE model predictions for the homoclinic regime using the SiLU activation function after $10,000$ training epochs. The plots show the predicted membrane voltage $V(t)$, gating variable $N(t)$, phase portrait, and training loss. Compared to earlier results with Tanh (see Figure~\ref{fig:homoclinic_node}), the use of SiLU leads to visibly improved dynamics, particularly in the reproduction of the orbit structure. However, significant deviations remain in the later phase and voltage trajectory, especially near the saddle-node region, indicating that the NODE still struggles to fully capture the global geometry of the homoclinic orbit.}
    \label{fig:homoclinicSiLU}
\end{figure}

In summary, the qualitative results across the three bifurcation regimes (Hopf, SNLC, and Homoclinic) demonstrate that both NODE and PINN architectures are capable of reproducing the dynamics of the Morris–Lecar system under varying levels of complexity—provided that a sufficient number of training epochs is allowed. In simpler bifurcation scenarios such as Hopf and SNLC, both models gradually align with the ground-truth trajectories. In the homoclinic regime, only the PINN—when extensively trained—initially manages to approximate the intricate dynamics, while the NODE model struggles to converge. However, as shown in Remark~\ref{rem:SiLU}, the use of the SiLU activation function significantly improves the NODE’s performance, enabling it to capture the orbit structure with greater fidelity in this challenging regime.
\subsection{Quantitative Evaluation}\label{sec4.2}

To systematically assess the predictive performance of both models, we conducted a quantitative evaluation across three representative dynamical regimes—Hopf, SNLC, and Homoclinic—using multiple training durations ranging from $1,000$ to $20,000$ epochs. The comparison relied on a comprehensive set of metrics: MSE, RMSE, MAE, Maximum Error, MAPE, Root RMSPE, and the $R^2$. Evaluation was performed separately for the two state variables $V$ and $N$, with results aggregated in Tables \ref{tab:PINN} and \ref{tab:NODE}.

As shown in Table \ref{tab:PINN}, the PINN achieved consistently superior accuracy across all scenarios and training durations. For instance, in the Hopf regime at $20,000$ epochs, the PINN attained an RMSE of just $0.72$ for $V$, and an $R^{2}$ score of $0.9996$—indicative of near-perfect reconstruction of the system dynamics. Similarly, in the SNLC case, the RMSE remained below $1.0$ for both variables across all training regimes, confirming robustness.

Conversely, Table \ref{tab:NODE} reveals that the NODE model exhibited higher sensitivity to both scenario complexity and training duration. While NODE achieved acceptable accuracy in the Hopf regime (e.g., RMSE $V = 1.04$ at $20,000$ epochs), it consistently underperformed in the SNLC and homoclinic cases. Notably, in the homoclinic scenario, NODE failed to reduce the RMSE below $6.0$ even after $20,000$ training epochs, and its $R^{2}$ value dropped significantly, indicating poor alignment with the reference trajectory.

Furthermore, across all scenarios, the maximum error and MAPE were markedly lower for the PINN, suggesting that the integration of physical constraints enhances generalization and reduces extreme deviations. These trends were particularly evident in stiff regimes such as the homoclinic bifurcation, where PINN’s inclusion of ODE residuals in the loss function preserved the temporal structure of spike events—a property that NODE struggled to replicate.

In summary, the quantitative results provide compelling evidence that physics-informed architectures not only accelerate convergence but also improve predictive fidelity across a wide range of dynamical behaviors. This is especially critical in regimes where long-term accuracy and stability are essential.

Figure \ref{fig:totalmse} illustrates the evolution of MSE across training epochs for both the NODE and PINN architectures in all three dynamical regimes. The PINN demonstrates a smooth and monotonic decrease in error, reflecting strong convergence properties and improved generalization as training progresses. In contrast, the NODE exhibits non-monotonic and unstable behavior—achieving temporary error reductions, followed by a marked deterioration, especially in the Homoclinic regime. 
\begin{figure}[!htbp]
    \centering
    \begin{subfigure}[t]{0.48\textwidth}
        \centering
        \includegraphics[width=\textwidth]{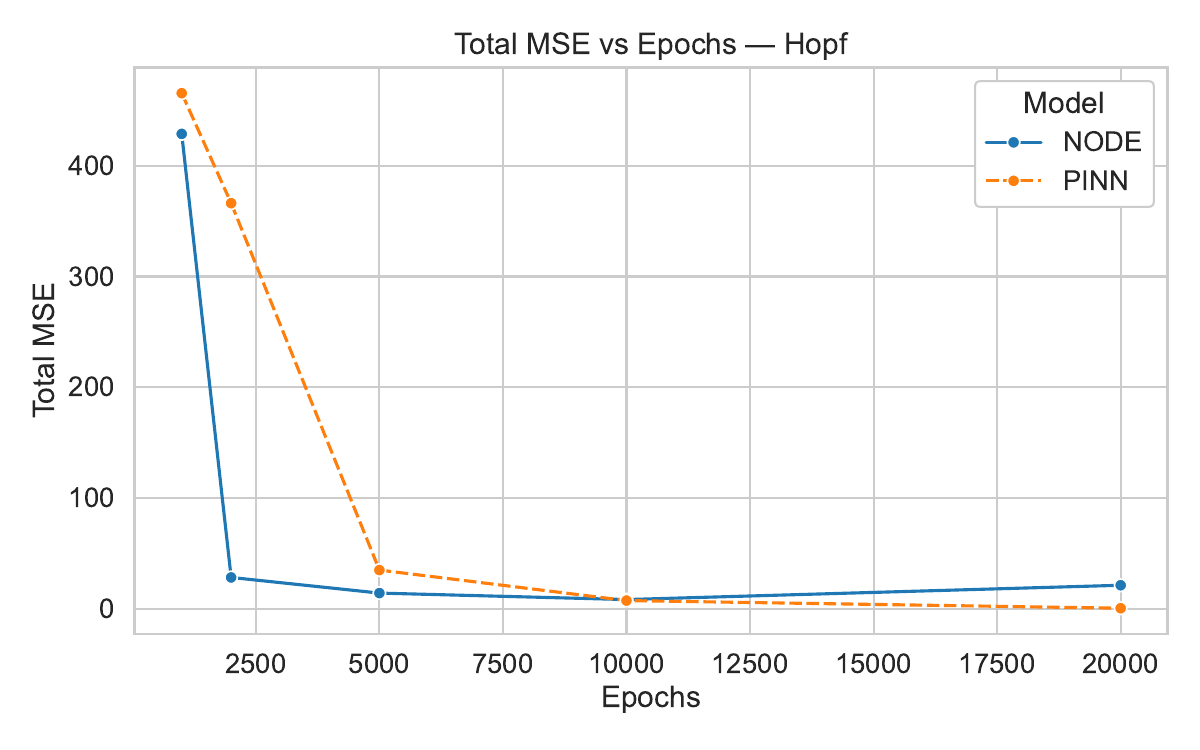}
        \caption{Hopf regime.}
    \end{subfigure}
    \hfill
    \begin{subfigure}[t]{0.48\textwidth}
        \centering
        \includegraphics[width=\textwidth]{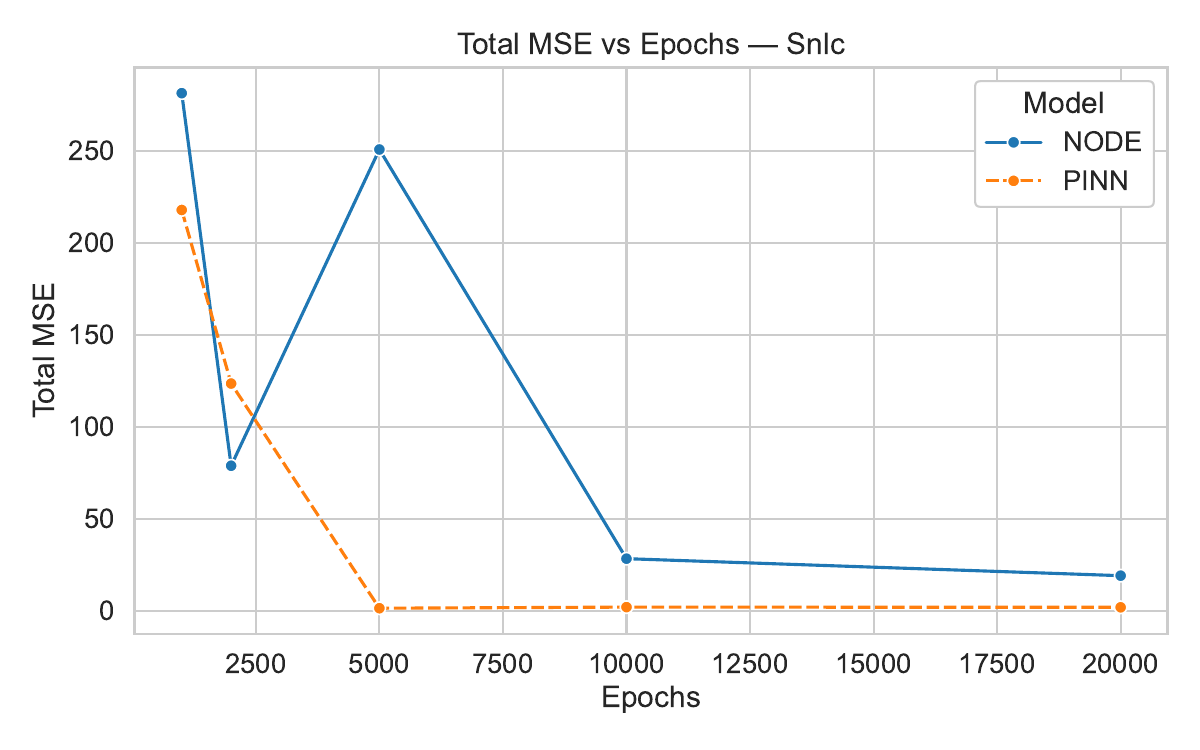}
        \caption{SNLC regime.}
    \end{subfigure}
    \begin{subfigure}[t]{0.48\textwidth}
        \centering
        \includegraphics[width=\textwidth]{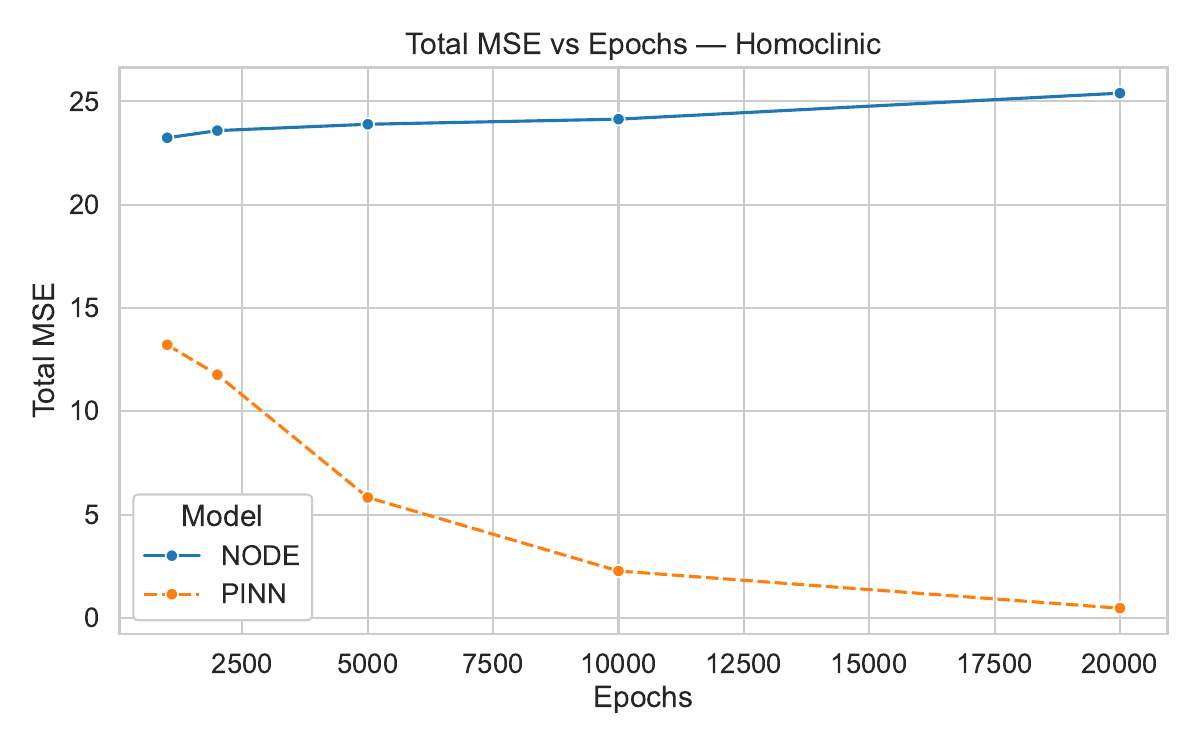}
        \caption{Homoclinic regime.}
    \end{subfigure}
		\caption{Total MSE across epochs for NODE and PINN models in the Hopf, SNLC, and Homoclinic regimes. The total MSE is computed as the sum of the MSE values for $V$ and $N$; PINN achieves lower error in all cases, especially in the Homoclinic regime.}
      \label{fig:totalmse}
\end{figure}

\subsection*{Interpretation of Metrics Across Architectures}

Although both PINNs and NODEs—produce predictions of the same physical variables $(V, N)$ in the Morris--Lecar system, their training procedures are fundamentally different and must be taken into account when interpreting the evaluation metrics.

Specifically, NODEs rely purely on data-driven learning: they learn a vector field $f_\theta(y)$ such that the numerical solution $\hat{y}(t)$ matches the observed trajectories, using supervised loss on scaled data. In contrast, PINNs approximate the solution $y_\theta(t)$ directly by minimizing both the data loss and a physics-based residual loss derived from the governing differential equations, and they do so without any data normalization.

To ensure a fair comparison, all evaluation MSE, RMSE, MA, $R^2$, MAPE, and RMSPE—are computed in physical units (after inverse scaling in the case of NODEs). Nevertheless, due to the differences in objective functions, regularization effects, and numerical sensitivities, these metrics reflect not only predictive accuracy but also the inherent biases of each learning paradigm. Therefore, performance differences must be interpreted in light of these methodological distinctions.

Table~\ref{tab:best_models_comparison} provides a comparative summary of the best-performing models for each regime, highlighting key performance metrics—including total MSE, RMSE, and training time—for both the PINN and NODE frameworks. This table encapsulates the optimal results achieved across all training settings, allowing for a direct evaluation of model effectiveness and computational cost.
\begin{table}[htbp]
\centering
\caption{Comparison of Best Performing Models (PINN vs NODE) per Scenario}
\label{tab:best_models_comparison}
\begin{tabularx}{\textwidth}{@{}l c c c c c c c@{}}
\toprule
\textbf{Scenario} & \textbf{Method} & \textbf{Epochs} & \textbf{Total MSE} & \textbf{RMSE\textsubscript{V}} & \textbf{RMSE\textsubscript{N}} & \textbf{Training Time (s)} \\
\midrule
\multirow{2}{*}{Hopf} 
  & PINN  & 20000  & 3.0577e-01 & 0.5529 & 0.0069 & 145.74   \\
  & NODE  & 10000 & 8.1303e+00 & 2.8514 &  0.0035 & 2857.54   \\
\midrule
\multirow{2}{*}{SNLC} 
  & PINN  & 5000 & 1.6022e+00 & 1.2656 & 0.0221 & 37.69   \\
  & NODE  & 20000  & 1.9270e+01 & 4.3898 & 0.0101 & 6143.12  \\
\midrule
\multirow{2}{*}{Homoclinic} 
  & PINN  & 20000  & 4.6681e-01 & 0.6832 &   0.0091 & 146.94    \\
  & NODE  & 1000  & 2.3236e+01 & 4.8199 & 0.0672 & 251.52 sec   \\
\bottomrule
\end{tabularx}
\end{table}


\section{Discussion and Conclusions}

This study undertook a detailed comparison between two modern frameworks for learning dynamical systems from data: PINNs~\cite{raissi2019, karniadakis2021, lu2021} and NODEs~\cite{chen2019, dupont2019, massaroli2021}. Using the Morris--Lecar model \cite{gasparinatou2022, morris1981,fenichel1979,ermentrout2010, li2011} as a benchmark, we investigated their performance across three distinct dynamical regimes: Hopf bifurcation, SNLC, and Homoclinic orbit.

The overall results indicate that the performance of each model is highly dependent on the complexity of the dynamics. In the relatively smooth Hopf regime, both PINNs and NODEs achieved good accuracy. The PINN model, however, demonstrated greater training efficiency and consistency. For instance, at $10,000$ epochs, the PINN reached a total MSE of only $8.059$ with a training time of $142$ seconds, whereas the NODE achieved a lower MSE of $5.661$ but required over $2800$ seconds. Notably, the PINN attained an $R^2_V$ of $0.998$, indicating an almost perfect reconstruction of voltage dynamics, with significantly smaller maximum errors.

In the SNLC regime, the system becomes more nonlinear and sensitive to parameter changes. Here, the advantage of the PINN approach becomes more pronounced. At $20,000$ epochs, the PINN achieved a total MSE of $4.044$, an RMSE of $1.573$, and an $R^2_V$ of $0.980$. The NODE, despite reducing its total MSE to $3.745$, exhibited severe trajectory distortions, as evidenced by a large $\mathrm{MAXERR}_V$ of $62.218$, compared to only $1.656$ for the PINN. Moreover, NODE training time increased drastically ($6489$ seconds), while the PINN remained below $40$ seconds. These results highlight that NODEs, though flexible, may fail to maintain trajectory fidelity without additional constraints.

The most challenging scenario was the homoclinic regime, characterized by slow-fast dynamics and bifurcation sensitivity. In this case, the PINN consistently preserved qualitative features such as the prolonged approach to the saddle-node and the escape dynamics. Quantitatively, the PINN reached a total MSE of $0.467$ at $20,000$ epochs, with $\mathrm{MAXERR}_V$ of just $1.715$ and a strong $R^2_V$ of $0.958$. In contrast, the NODE failed to converge meaningfully: even at $20,000$ epochs with advanced activation functions (e.g., SiLU, remark~\ref{rem:SiLU}), it remained at a high total MSE of $23.236$ and a $\mathrm{MAXERR}_V$ of $5.665$. These shortcomings suggest that NODEs struggle in regimes where trajectory stability is crucial, possibly due to the absence of a physics-informed inductive bias.

A closer look at the evolution of metrics across epoch settings (Appendix Tables~\ref{tab:PINN} and~\ref{tab:NODE}) confirms that PINNs benefit from progressive refinement with training, particularly in smoother regimes. In stiff regimes like homoclinic, the PINN retains relative stability and physical consistency, whereas the NODE demonstrates erratic convergence, high sensitivity to initialization, and increased computational cost.

In terms of efficiency, PINNs proved significantly more computationally economical. Across all scenarios and epochs, the training time for PINNs remained under $150$ seconds, while NODEs often exceeded $1000$ seconds and reached up to $6489$ seconds in the SNLC regime at $20,000$ epochs—despite no corresponding gain in accuracy.

In conclusion, this study demonstrates that PINNs are particularly well-suited for modeling systems where the governing equations are known and the dynamics are stiff or complex. The embedded physical constraints guide the network toward plausible solutions, offering robustness, stability, and lower computational burden~\cite{karniadakis2021}.

By contrast, NODEs maintain a key advantage: they can learn directly from data without prior knowledge of the system's structure. This flexibility is valuable in scenarios where the physics is partially understood, highly noisy, or prohibitively complex~\cite{rackauckas2021, kidger2020}. Nevertheless, their application to stiff biological models remains challenging, requiring careful normalization, architectural tuning, or external regularization.

Overall, our results highlight the critical role of inductive bias in learning complex biological dynamics. When prior knowledge is available, PINNs provide robust and efficient solutions. NODEs, although more general, face challenges in stiff regimes without proper constraints. 

\section{Limitations and Future Work}
\label{lim}
While this study provides a focused comparison between PINNs and NODEs using the Morris–Lecar model, it is subject to several limitations. First, the evaluation is based on a low-dimensional (two-dimensional) neuronal system. Although this setting allows for analytical tractability, it does not capture the complexity of higher-dimensional or spatially distributed biological models. Second, both methods were implemented with fixed hyperparameters and standard fully connected architectures. The use of more advanced neural designs—such as residual connections, recurrent units, or attention mechanisms—could significantly affect model performance. Third, the NODE models were trained on noise-free, uniformly sampled data. Such idealized conditions simplify the learning task compared to real-world scenarios, where data may be noisy, incomplete, or sampled at irregular intervals.

Although training time and floating-point operation counts were reported, a detailed analysis of computational scalability and solver efficiency lies beyond the scope of this study. These limitations suggest several directions for future research, including evaluations on more complex systems, incorporation of noisy or sparse data, and the exploration of adaptive or hybrid modeling strategies. The development of hybrid approaches that combine the respective strengths of PINNs and NODEs represents a particularly promising avenue for future investigation.
\begin{appendices}
\begin{sidewaystable}
\scriptsize
\section{}\label{secA1}
\caption{Detailed PINN Performance Across All Epochs and Scenarios}
\label{tab:PINN}
\begin{tabular}{@{}l r r r r r r r r r r r@{}}
\toprule
\textbf{Scenario} & \textbf{Epochs} & \textbf{Total MSE} & \textbf{$\mathrm{MAPE}_V$} & \textbf{$\mathrm{MAPE}_N$} &
\textbf{$\mathrm{MAE}_V$} & \textbf{$\mathrm{MAE}_N$} & \textbf{$R^2_V$} & \textbf{$R^2_N$} &
\textbf{$\mathrm{MAXERR}_V$} & \textbf{$\mathrm{MAXERR}_N$} & \textbf{Time (s)} \\
\midrule
             hopf &            1000 &            465.644 &               88695577.180 &                 993439.777 &                    16.260 &                     0.098 &            0.300 &           -0.210 &                       56.598 &                        0.371 &             6.890 \\
             hopf &            2000 &            366.243 &               14151248.974 &                 896333.463 &                    12.922 &                     0.077 &            0.449 &            0.149 &                       57.480 &                        0.375 &            15.134 \\
             hopf &            5000 &             34.774 &                2402399.507 &                  42864.153 &                     3.744 &                     0.032 &            0.948 &            0.828 &                       17.441 &                        0.231 &            36.431 \\
             hopf &           10000 &              7.202 &                 705903.823 &                  33737.237 &                     1.616 &                     0.014 &            0.989 &            0.977 &                        9.834 &                        0.080 &            73.688 \\
             hopf &           20000 &              0.306 &                 241962.693 &                  47727.391 &                     0.271 &                     0.004 &            1.000 &            0.997 &                        3.303 &                        0.036 &           145.744 \\
             snlc &            1000 &            217.796 &                 823407.161 &                  27580.982 &                     8.915 &                     0.042 &            0.405 &            0.355 &                       55.646 &                        0.397 &             8.091 \\
             snlc &            2000 &            123.560 &                5283144.839 &                 215580.793 &                     6.967 &                     0.047 &            0.662 &            0.460 &                       45.254 &                        0.331 &            15.284 \\
             snlc &            5000 &              1.602 &                1133676.703 &                  62993.930 &                     0.437 &                     0.010 &            0.996 &            0.962 &                       10.693 &                        0.180 &            37.686 \\
             snlc &           10000 &              2.206 &                 312527.248 &                  48876.481 &                     0.570 &                     0.008 &            0.994 &            0.961 &                       10.935 &                        0.167 &            76.209 \\
             snlc &           20000 &              2.100 &                 572976.069 &                  17269.348 &                     0.525 &                     0.006 &            0.994 &            0.980 &                       12.023 &                        0.125 &           162.772 \\
       homoclinic &            1000 &             13.218 &               32313362.091 &                 662893.201 &                     2.080 &                     0.028 &            0.507 &            0.453 &                       16.871 &                        0.204 &             8.152 \\
       homoclinic &            2000 &             11.769 &                3834142.861 &                 600851.690 &                     2.167 &                     0.035 &            0.561 &            0.410 &                       14.528 &                        0.181 &            14.604 \\
       homoclinic &            5000 &              5.827 &                 592331.496 &                 150909.513 &                     1.337 &                     0.018 &            0.783 &            0.764 &                        9.545 &                        0.125 &            38.434 \\
       homoclinic &           10000 &              2.269 &                 852403.136 &                 131646.454 &                     0.859 &                     0.013 &            0.915 &            0.896 &                        6.932 &                        0.087 &            76.320 \\
       homoclinic &           20000 &              0.467 &                  76218.942 &                  36291.641 &                     0.404 &                     0.006 &            0.983 &            0.981 &                        2.932 &                        0.035 &           146.939 \\
\bottomrule
\end{tabular}
\end{sidewaystable}

\begin{sidewaystable}
\scriptsize
\caption{Detailed NODE Performance Across All Epochs and Scenarios}
\label{tab:NODE}
\begin{tabular}{@{}l r r r r r r r r r r r@{}}
\toprule
\textbf{Scenario} & \textbf{Epochs} & \textbf{Total MSE} & \textbf{$\mathrm{MAPE}_V$} & \textbf{$\mathrm{MAPE}_N$} &
\textbf{$\mathrm{MAE}_V$} & \textbf{$\mathrm{MAE}_N$} & \textbf{$R^2_V$} & \textbf{$R^2_N$} &
\textbf{$\mathrm{MAXERR}_V$} & \textbf{$\mathrm{MAXERR}_N$} & \textbf{Time (s)} \\
\midrule
             hopf &            1000 &            428.841 &                    214.801 &                     48.470 &                    15.299 &                     0.090 &            0.355 &            0.310 &                       53.275 &                        0.276 &           232.181 \\
             hopf &            2000 &             28.105 &                     65.192 &                      5.342 &                     2.984 &                     0.010 &            0.958 &            0.993 &                       31.270 &                        0.028 &           488.823 \\
             hopf &            5000 &             13.974 &                     36.552 &                      2.637 &                     1.852 &                     0.004 &            0.979 &            0.998 &                       25.789 &                        0.015 &          1293.660 \\
             hopf &           10000 &              8.130 &                     38.001 &                      1.636 &                     1.143 &                     0.003 &            0.988 &            0.999 &                       22.397 &                        0.013 &          2857.542 \\
             hopf &           20000 &             21.052 &                     40.509 &                      4.739 &                     2.991 &                     0.010 &            0.968 &            0.991 &                       22.901 &                        0.048 &          5833.402 \\
             snlc &            1000 &            281.210 &                    143.535 &                    527.669 &                    11.403 &                     0.079 &            0.231 &           -0.061 &                       58.137 &                        0.387 &           249.012 \\
             snlc &            2000 &             78.934 &                     61.774 &                    130.922 &                     6.530 &                     0.021 &            0.784 &            0.921 &                       27.010 &                        0.144 &           542.979 \\
             snlc &            5000 &            250.626 &                    124.842 &                    412.268 &                    10.407 &                     0.071 &            0.315 &            0.105 &                       58.191 &                        0.371 &          1302.418 \\
             snlc &           10000 &             28.500 &                     25.000 &                     38.475 &                     3.524 &                     0.007 &            0.922 &            0.990 &                       22.275 &                        0.042 &          2904.200 \\
             snlc &           20000 &             19.270 &                     22.964 &                     52.490 &                     2.957 &                     0.007 &            0.947 &            0.992 &                       19.262 &                        0.047 &          6143.116 \\
       homoclinic &            1000 &             23.236 &                     83.023 &                     16.846 &                     2.633 &                     0.035 &            0.134 &           -0.058 &                       23.078 &                        0.347 &           251.524 \\
       homoclinic &            2000 &             23.586 &                     81.253 &                     16.294 &                     2.663 &                     0.033 &            0.121 &            0.146 &                       22.201 &                        0.260 &           566.709 \\
       homoclinic &            5000 &             23.893 &                     79.025 &                     16.108 &                     2.671 &                     0.032 &            0.109 &            0.224 &                       21.820 &                        0.225 &          1708.793 \\
       homoclinic &           10000 &             24.140 &                     80.213 &                     16.471 &                     2.684 &                     0.032 &            0.100 &            0.224 &                       21.934 &                        0.220 &          3142.986 \\
       homoclinic &           20000 &             25.399 &                     80.567 &                     17.618 &                     2.936 &                     0.034 &            0.053 &            0.153 &                       19.920 &                        0.228 &          5201.115 \\
\bottomrule
\end{tabular}
\end{sidewaystable}
\end{appendices}
\clearpage 

\section*{Declarations}

\paragraph{Availability of data and materials} All data and source code used in this study are available on GitHub: \url{https://github.com/nikmatz/Morris-lecar_NODE_vs_PINNS}.

\paragraph{Competing interests} The authors declare that they have no known competing financial interests or personal relationships that could have appeared to influence the work reported in this paper.

\paragraph{Funding} Not applicable.

\paragraph{Authors' contributions} NM designed and performed all experiments, analyzed the results, and wrote the manuscript. SF contributed to early discussions and reviewed the final version. All authors read and approved the final manuscript.

\bibliographystyle{plainnat}
\bibliography{Comparing_Physics_Informed_and_Neura_ODE}
\end{document}